\documentclass[11pt,reqno]{amsart}

\usepackage[all]{xy}
\usepackage{amssymb}
\usepackage{mathrsfs}
\usepackage{cite}
\usepackage{amsfonts}
\newcommand{\rmv}[1]{}

\def\wh{\widehat}

\def\inv{^{-1}}
\def\p{\varphi}

\def\D{\mathrel{\mathscr D}} % D - relation

\def\R{\mathrel{{\mathscr R}}} % R - relation
\def\e<{\leq _{E}}

\def\ov#1{\ensuremath{\overline {#1}}}

\def\til#1{\ensuremath{\widetilde {#1}}}

\def\malce{\mathbin{\hbox{$\bigcirc$\rlap{\kern-8.3pt\raise0,50pt\hbox{$\mathtt{m}$}}}}}
\def\CC{\mathbb C}

\def\Ind{\mathrm{Ind}}

\def\Res{\mathrm{Res}}
\def\1sk{^{(1)}}

\def\to{\rightarrow}

\def\Thmname{Theorem}
\def\Propname{Proposition}
\def\Lemmaname{Lemma}
\def\Definitionname{Definition}
\newtheorem{Thm}{\Thmname}[section]
\newtheorem{Prop}[Thm]{\Propname}
\newtheorem{Lemma}[Thm]{\Lemmaname}
{\theoremstyle{definition}
\newtheorem{Def}[Thm]{\Definitionname}}
{\theoremstyle{remark}
\newtheorem{Rmk}[Thm]{Remark}}
\newtheorem{Cor}[Thm]{Corollary}
{\theoremstyle{remark}
\newtheorem{Example}[Thm]{Example}}

{\theoremstyle{remark}
}

{\theoremstyle{remark}
\newtheorem{Step}{Step}}

\numberwithin{equation}{section}

\title{A Groupoid Approach to Discrete Inverse Semigroup Algebras}

\author{Benjamin Steinberg}
\address{School of Mathematics and Statistics\\
Carleton University \\
1125 Colonel By Drive\\
Ottawa, Ontario  K1S 5B6 \\
Canada}
\thanks{The author was supported in part by NSERC.  He also gratefully acknowledges the support of DFG. Some of this work was done while the author was visiting Bar-Ilan in the summer of 2008}
\email{bsteinbg@math.carleton.ca}
\date{March 20, 2009}

\begin{document}
\begin{abstract}
Let $K$ be a commutative ring with unit and $S$ an inverse semigroup.  We show that the semigroup algebra $KS$ can be described as a convolution algebra of functions on the universal \'etale groupoid associated to $S$ by Paterson.  This result is a simultaneous generalization of the author's earlier work on finite inverse semigroups and Paterson's theorem for the universal $C^*$-algebra.  It provides a convenient topological framework for understanding the structure of $KS$, including the center and when it has a unit.  In this theory, the role of Gelfand duality is replaced by Stone duality.

Using this approach we are able to construct the finite dimensional irreducible representations of an inverse semigroup over an arbitrary field as induced representations from associated groups, generalizing the well-studied case of an inverse semigroup with finitely many idempotents.    More generally, we describe the irreducible representations of an inverse semigroup $S$ that can be induced from associated groups as precisely those satisfying a certain ``finiteness condition''.   This ``finiteness condition'' is satisfied, for instance, by all representations of an inverse semigroup whose image contains a primitive idempotent.
\end{abstract}

\maketitle
\tableofcontents

\section{Introduction}
It is by now well established in the $C^*$-algebra community that there is a close relationship between inverse semigroup $C^*$-algebras and \'etale groupoid $C^*$-algebras~\cite{Paterson,Exel,Renault,graphinverse,ultragraph,higherrank,resendeetale,strongmorita}.  More precisely, Paterson assigned to each inverse semigroup $S$ an \'etale (in fact, ample) groupoid $\mathscr G(S)$, called its universal groupoid, and showed that the universal and reduced $C^*$-algebras of $S$ and $\mathscr G(S)$ coincide~\cite{Paterson}.  On the other hand, if $\mathscr G$ is a discrete groupoid and $K$ is a unital commutative ring, then there is an obvious way to define a groupoid algebra $K\mathscr G$.  The author showed that if $S$ is an inverse semigroup with finitely many idempotents, then $KS\cong K\mathscr G$ for the so-called underlying groupoid $\mathscr G$ of $S$~\cite{mobius1,mobius2}; this latter groupoid coincides with the universal groupoid $\mathscr G(S)$ when $S$ has finitely many idempotents.  It therefore seems natural to conjecture that, for any inverse semigroup $S$, one has that $KS\cong K\mathscr G(S)$ for an appropriate definition of $K\mathscr G(S)$.  This is what we achieve in this paper.  We then proceed to use groupoids to establish a number of new results about inverse semigroup algebras, including a description of all the finite dimensional irreducible representations over a field as induced representations from groups, as well as of all simple modules over an arbitrary commutative ring satisfying a certain additional condition.

The idea behind our approach is to view $K$ as a topological ring by endowing it with the discrete topology.  One can then imitate the usual definition of continuous functions with compact support on an \'etale groupoid (note that we do not assume that groupoids are Hausdorff, so one must take the usual care with this). It turns out that \'etale groupoids are too general to deal with in the discrete context because they do not have enough continuous functions with compact support. However, the category of ample groupoids from~\cite{Paterson} is just fine for the task.  These groupoids have a basis of compact open sets and so they have many continuous maps to discrete spaces.  A key idea is that the role of Gelfand duality for commutative $C^*$-algebras can now be played by Stone duality for boolean rings.  Paterson's universal groupoid $\mathscr G(S)$ is an ample groupoid, and so fits nicely into this context.

The paper proceeds as follows.  First we develop the basic theory of the convolution algebra $K\mathscr G$ of an ample groupoid $\mathscr G$ over a commutative ring with unit $K$, including a description of its center.  We then discuss ample actions of inverse semigroups and the groupoid of germs construction.   The algebra of the groupoid of germs is shown to be a cross product (in an appropriate sense) of a commutative $K$-algebra with the inverse semigroup in the Hausdorff case.  The universal groupoid is constructed as the groupoid of germs of the spectral action of $S$, which is shown to be the terminal object in the category of boolean actions via Stone duality.  It is proved that the universal groupoid is Hausdorff if and only if the intersection of principal downsets in $S$ is finitely generated as a downset, yielding the converse to a result of Paterson~\cite{Paterson}.  The isomorphism of $KS$ with $K\mathscr G(S)$ is then established via the M\"obius inversion trick from~\cite{mobius1,mobius2}; Paterson's result for the universal $C^*$-algebra is obtained as a consequence of the case $K=\mathbb C$ via the Stone-Weierstrass theorem.  We believe the proof to be easier than Paterson's since we avoid the ardous detour through a certain auxiliary inverse semigroup that Paterson follows in order to use the theory of localizations~\cite{Paterson}.  Using the isomorphism with $KS$ with $K\mathscr G(S)$ we give a topological proof of a result of Crabb and Munn describing the center of the algebra of a free inverse monoid~\cite{Munncentre}.

To study simple $K\mathscr G$-modules, we associate to each unit $x$ of $\mathscr G$ induction and restriction functors between the categories of $KG_x$-modules and $K\mathscr G$-modules, where $G_x$ is the isotropy group of $x$.  It turns out that induction preserves simplicity, whereas restriction takes simple modules to either $0$ or simple modules.
This allows us to obtain an explicit parameterization of the finite dimensional simple $K\mathscr G$-modules as induced representations from isotropy groups in the case that $K$ is a field.  More generally, we can construct all simple $K\mathscr G$-modules satisfying an additional finiteness condition as induced representations for $K$ an arbitrary commutative ring with unit.  The methods are reminiscent of the theory developed by Munn and Ponizovsky for finite semigroups~\cite{CP,oknisemigroupalgebra}, as interpreted through~\cite{myirreps}.

The final section of the paper applies the results to inverse semigroups via their universal groupoids.  In particular, we parameterize all the finite dimensional simple $KS$-modules for an inverse semigroup $S$ in terms of group representations.  Munn gave a different construction of the finite dimensional irreducible representations of an arbitrary inverse semigroup~\cite{Munnarb} via cutting to ideals and reducing to the case of $0$-simple inverse semigroups; it takes him a bit of argument to deduce the classical statements for semigroups with finitely many idempotents from this approach.   We state our result in a groupoid-free way, although the proof relies on groupoids.  As a corollary we give necessary and sufficient conditions for the finite dimensional irreducible representations to separate points of $S$. Our techniques also construct all the simple $KS$-modules for inverse semigroups $S$ satisfying descending chain condition on idempotents, or whose idempotents are central or form a descending chain isomorphic to $(\mathbb N,\geq)$.

\section{\'Etale and ample groupoids}
By a \emph{groupoid} $\mathscr G$, we mean a small category in which every arrow is an isomorphism.  Objects will be identified with the corresponding units and the space of units will be denoted $\mathscr G^0$.  Then, for $g\in \mathscr G$, the domain and range maps are given by $d(g)=g\inv g$ and $r(g)=gg\inv$, respectively.  A \emph{topological groupoid} is a groupoid whose underlying set is equipped with a topology making the product and inversion continuous (where the set of composable pairs is given the induced topology from the product topology).

In this paper, we follow the usage of Bourbaki and reserve the term compact to mean a Hausdorff space with the Heine-Borel property.  Notice that a locally compact space need not be Hausdorff. By a \emph{locally compact groupoid}, we mean a topological groupoid $\mathscr G$ that is locally compact and whose unit space $\mathscr G^0$ is locally compact Hausdorff in the the induced topology.
A locally compact groupoid $\mathscr G$ is said to be \emph{\'etale} if  the domain map $d\colon \mathscr G\to \mathscr G^0$ is \'etale, that is, a local homeomorphism.  We do not assume that $\mathscr G$ is Hausdorff.    For basic properties of \'etale groupoids (also called $r$-discrete groupoids), we refer to the treatises~\cite{Exel,Paterson,Renault}.  We principally follow~\cite{Exel} in terminology.  Fix an \'etale groupoid $\mathscr G$ for this section.  A basic property of \'etale groupoids is that their unit space is open~\cite[Proposition 3.2]{Exel}.

\begin{Prop}\label{unitsareopen}
The subspace $\mathscr G^0$ is open in $\mathscr G$.
\end{Prop}

Of critical importance is the notion of a slice (or $\mathscr G$-set, or local bissection).

\begin{Def}[Slice]
A \emph{slice} $U$ is an open subset of $\mathscr G$ such that $d|_U$ and $r|_U$ are injective (and hence homeomorphisms since $d$ and $r$ are open).  The set of all slices of $\mathscr G$ is denoted $\mathscr G^{op}$.
\end{Def}

One can view a slice as the graph of a partial homeomorphism between $d(U)$ and $r(U)$ via the topological embedding $U\hookrightarrow d(U)\times r(U)$ sending $u\in U$ to $(d(u),r(u))$.  Notice that any slice is locally compact Hausdorff in the induced topology, being homeomorphic to a subspace of $\mathscr G^0$.

An \emph{inverse semigroup} is a semigroup $S$ so that, for all $s\in S$, there exists a unique $s^*\in S$ so that $ss^*s=s$ and $s^*ss^*=s^*$.  The set $E(S)$ of idempotents of $S$ is a commutative subsemigroup; it is ordered by $e\leq f$ if and only if $ef=e$.  With this ordering $E(S)$ is a meet semilattice with the meet given by the product.  Hence, it is often referred to as the \emph{semilattice of idempotents} of $S$.  The order on $E(S)$ extends to $S$ as the so-called \emph{natural partial order} by putting $s\leq t$ if $s=et$ for some idempotent $e$ (or equivalently $s=tf$ for some idempotent $f$).   This is equivalent to $s=ts^*s$ or $s=ss^*t$.  If $e\in E(S)$, then the set $G_e=\{s\in S\mid ss^*=e=s^*s\}$ is a group, called the \emph{maximal subgroup} of $S$ at $e$.  Idempotents $e,f$ are said to be $\mathscr D$-equivalent, written $e\D f$, if there exists $s\in S$ so that $e=s^*s$ and $f=ss^*$; this is the analogue of von Neumann-Murray equivalence.  See the book of Lawson~\cite{Lawson} for details.

\begin{Prop}
The slices form a basis for the topology of $\mathscr G$.  The set $\mathscr G^{op}$ is an inverse monoid under setwise multiplication.  The inversion is also setwise and the natural partial order is via inclusion.  The semilattice of idempotents is the topology of $\mathscr G^0$.
\end{Prop}
\begin{proof}
See~\cite[Propositions 3.5 and 3.8]{Exel}.
\end{proof}

A particularly important class of \'etale groupoids is that of ample groupoids~\cite{Paterson}.

\begin{Def}[Ample groupoid]
An \'etale groupoid is called \emph{ample} if the compact slices form a basis for its topology.
\end{Def}

One can show that the compact slices also form an inverse semigroup~\cite{Paterson}.  The inverse semigroup of compact slices is denoted $\mathscr G^{a}$.  The idempotent set of $\mathscr G^a$ is the semilattice of compact open subsets of $\mathscr G^0$.  Notice that if $U\in \mathscr G^a$, then any clopen subset $V$ of $U$ also belongs to $\mathscr G^a$.

Since we shall be interested in continuous functions with compact support into discrete rings, we shall restrict our attention to ample groupoids in order to ensure that we have ``enough'' continuous functions with compact support.  So from now on $\mathscr G$ is an ample groupoid.  To study ample groupoids it is convenient to discuss generalized boolean algebras and Stone duality.

\begin{Def}[Generalized boolean algebra]
A \emph{generalized boolean algebra} is a poset $P$ admitting finite (including empty) joins and non-empty finite meets so that the meet distributes over the join and if $a\leq b$, then there exists $x\in P$ so that $a\wedge x=0$ and $a\vee x=b$ where $0$ is the bottom of $P$.  Then, given $a,b\in P$ one can define the relative complement $a\setminus b$ of $b$ in $a$ to be the unique element $x\in P$ so that $(a\wedge b)\vee x=a$ and $a\wedge b\wedge x=0$.   Morphisms of generalized boolean algebras are expected to preserve finite joins and finite non-empty meets.  A generalized booleam algebra with a maximum (i.e., empty meet) is called a \emph{boolean algebra}.
\end{Def}

It is well known that a generalized boolean algebra is the same thing as a boolean ring.  A \emph{boolean ring} is a ring $R$ with idempotent multiplication. Such rings are automatically commutative of characteristic $2$.  The multiplicative semigroup of $R$ is then a semilattice, which is in fact a generalized boolean algebra.  The join is given by $a\vee b = a+b-ab$ and the relative complement by $a\setminus b = a-ab$.  Conversely, if $B$ is a generalized boolean algebra, we can place a boolean ring structure on it by using the meet as multiplication and the symmetric difference $a+b=(a\setminus b)\vee (b\setminus a)$ as the addition.  Boolean algebras correspond in this way to unital boolean rings.  For example, $\{0,1\}$ is a boolean algebra with respect to its usual ordering.  The corresponding boolean ring is the two-element field $\mathbb F_2$. See~\cite{halmosnew} for details.

\begin{Def}[Locally compact boolean space]
A Hausdorff space $X$ is called a \emph{locally compact boolean space} if it has a basis of compact open sets~\cite{halmosnew}.
\end{Def}

It is easy to see that the set $B(X)$ of compact open subspaces of any Hausdorff space $X$ is a generalized boolean algebra (and is a boolean algebra if and only if $X$ is compact). Restriction to the case of locally compact boolean spaces gives all generalized boolean algebras.  In detail, if $A$ is a generalized boolean algebra and $\mathrm{Spec}(A)$ is the set of non-zero morphisms $A\to \{0,1\}$ endowed with the subspace topology from $\{0,1\}^A$, then $\mathrm{Spec}(A)$ is a locally compact boolean space with $B(\mathrm{Spec}(A))\cong A$.  Dually, if $X$ is a locally compact boolean space, then $X\cong \mathrm{Spec}(B(X))$.  In fact, $B$ and $\mathrm{Spec}$ give a duality between the categories of locally compact boolean spaces with proper continuous maps and generalized boolean algebras:  this is the famous Stone duality. If $\psi\colon X\to Y$ is a proper continuous map of locally compact boolean spaces, then $\psi\inv\colon B(Y)\to B(X)$ is a homomorphism of generalized boolean algebras; recall that a continuous map is \emph{proper} if the preimage of each compact set is compact. If $\p\colon A\to A'$ is a morphism of generalized boolean algebras, then $\wh{\p}\colon \mathrm{Spec}(A')\to \mathrm{Spec}(A)$ is given by $\psi\mapsto \psi\p$.  The homeomorphism $X\to \mathrm{Spec}(B(X))$ is given by $x\mapsto \p_x$ where $\p_x(U) = \chi_U(x)$.  The isomorphism $A\to B(\mathrm{Spec}(A))$ sends $a$ to $D(a)=\{\p\mid \p(a)=1\}$.
The reader is referred to~\cite{halmosnew,stonespace2} for further details.

A key example for us arises from the consideration of ample groupoids.  If $\mathscr G$ is an ample groupoid, then $\mathscr G^0$ is a locally compact boolean space and $B(\mathscr G^0)=E(\mathscr G^a)$.  In fact, one has the following description of ample groupoids.

\begin{Prop}
An \'etale groupoid $\mathscr G$ is ample if and only if $\mathscr G^0$ is a locally compact boolean space.
\end{Prop}
\begin{proof}
If $\mathscr G$ is ample, we already observed that $\mathscr G^0$ is a locally compact boolean space.  For the converse, since $\mathscr G^{op}$ is a basis for the topology it suffices to show that each $U\in \mathscr G^{op}$ is a union of compact slices.  But $U$ is homeomorphic to $d(U)$ via $d|_U$.  Since $\mathscr G^0$ is a locally compact boolean space, we can write $d(U)$ as a union of compact open subsets of $\mathscr G^0$ and hence we can write $U$ as union of compact open slices by applying $d|_U\inv$.
\end{proof}

In any poset $P$, it will be convenient to use, for $p\in P$, the notation
\begin{align*}
p^{\uparrow}&=\{q \in P\mid q\geq p\}\\
p^{\downarrow}&=\{q\in P\mid q\leq p\}.
\end{align*}

\begin{Def}[Semi-boolean algebra]
A poset $P$ is called a \emph{semi-boolean algebra} if each principal downset $p^{\downarrow}$ with $p\in P$ is a boolean algebra.
\end{Def}

It is immediate that every generalized boolean algebra is a semi-boolean algebra.  A key example for us is the inverse semigroup $\mathscr G^a$ for an ample groupoid $\mathscr G$.

\begin{Prop}\label{semibooleanalgebra}
Let $\mathscr G$ be an ample groupoid.  Then $\mathscr G^a$ is a semi-boolean algebra.  Moreover, the following are equivalent:
\begin{enumerate}
\item $\mathscr G$ is Hausdorff;
\item  $\mathscr G^a$ is closed under pairwise intersections;
\item $\mathscr G^a$ is closed under relative complements.
\end{enumerate}
\end{Prop}
\begin{proof}
Let $U\in \mathscr G^a$.  Then the map $d\colon U\to d(U)$ gives an isomorphism between the posets $U^{\downarrow}$ and $B(d(U))$.  Since $B(d(U))$ is a boolean algebra, this proves the first statement.  Suppose that $\mathscr G$ is Hausdorff and $U,V\in \mathscr G^a$.  Then $U\cap V$ is a clopen subset of $U$ and hence belongs to $\mathscr G^a$.  If $\mathscr G^a$ is closed under pairwise intersections and $U,V\in \mathscr G^a$, then $U\cap V$ is compact open and so $U\cap V$ is clopen in $U$.  Then $U\setminus V = U\setminus (U\cap V)$ is a clopen subset of $U$ and hence belongs to $\mathscr G^a$.  Finally, suppose that $\mathscr G^a$ is closed under relative complements and let $g,h\in \mathscr G$.  As $\mathscr G^a$ is a basis for the topology on $\mathscr G$, we can find slices $U,V\in \mathscr G^a$ with $g\in U$ and $h\in V$.  If $g,h\in U$ or $g,h\in V$, then we can clearly separate them by disjoint open sets since $U$ and $V$ are Hausdorff.  Otherwise, $g\in U\setminus V$, $h\in V\setminus U$ and these are disjoint open sets as $\mathscr G^a$ is closed under relative complements.  This completes the proof.
\end{proof}

\section{The algebra of an ample groupoid}
Fix for this section an ample groupoid $\mathscr G$.
Following the idea of Connes~\cite{Connes}, we now define the space of continuous $K$-valued functions with compact support on $\mathscr G$ where $K$ is a commutative ring with unit.

\begin{Def}[$K\mathscr G$]\label{definecompactsupport}
If $\mathscr G$ is an ample groupoid and $K$ is a commutative ring with unit equipped with the discrete topology, then $K\mathscr G$ is the space of all $K$-valued functions on $\mathscr G$ spanned by functions $f\colon \mathscr G\to K$ such that:
\begin{enumerate}
\item There is an open Hausdorff subspace $V$ in $\mathscr G$ so that $f$ vanishes outside $V$;
\item $f|_V$ is continuous with compact support.
\end{enumerate}
We call $K\mathscr G$ the algebra of continuous $K$-valued functions on $\mathscr G$ with compact support (but the reader is cautioned that if $\mathscr G$ is not Hausdorff, then $K\mathscr G$ will contain discontinuous functions).
\end{Def}

For example, if $\mathscr G$ has the discrete topology, then one can identify $K\mathscr G$ with the vector space of all functions of finite support on $\mathscr G$.  A basis then consists of the functions $\delta_g$ with $g\in \mathscr G$. In general, as $K$ is discrete, the support of a function $f$  as in (1) will be, in fact, compact open and so one may take $V$ to be the support.

\begin{Prop}\label{supportisopen}
With the notation of Definition~\ref{definecompactsupport}, one can always choose $V$ to be compact open so that $\mathrm{supp}(f)=V=f^{-1}(K\setminus \{0\})$.
\end{Prop}
\begin{proof}
Let $f$ and $V$ be as in (1) and (2).
Since $K$ is discrete and $f(\mathrm{supp}(f|_V))$ is compact, it must in fact be finite.  Hence $f^{-1}(f(V)\setminus \{0\})$ is a clopen subset of $V$ contained in $\mathrm{supp}(f)$ and so is compact open.  It follows that $\mathrm{supp}(f|_V)=f^{-1}(K\setminus \{0\})$ is compact open and may be used in place of $V$ in (1) of Definition~\ref{definecompactsupport}.
\end{proof}

Notice that if $\mathscr G$ is not Hausdorff, it will have compact open subsets that are not closed (cf.~Proposition~\ref{semibooleanalgebra}).  The corresponding characteristic function of such a compact open will be discontinuous, but belong to $K\mathscr G$.
It turns out that the algebraic structure of $K\mathscr G$ is controlled by $\mathscr G^a$.  We start at the level of $K$-modules.

\begin{Prop}\label{characteristicbasis}
The space $K\mathscr G$ is spanned by the characteristic functions of elements of $\mathscr G^a$.
\end{Prop}
\begin{proof}
Evidently, if $U\in \mathscr G^a$, then $\chi_U\in K\mathscr G$.
Let $A$ be the subspace spanned by such characteristic functions.  By Proposition~\ref{supportisopen}, it suffices to show that if $f\colon \mathscr G\to K$ is a function so that $V=f\inv(K\setminus \{0\})$ is compact open and $f|_V$ is continuous, then $f\in A$.  Since $K$ is discrete and $V$ is compact, we have $f(V)\setminus \{0\}=\{c_1,\ldots,c_r\}$ for certain $c_i\in K\setminus \{0\}$ and the $V_i=f^{-1}(c_i)$, for $i=1,\ldots, r$, are disjoint compact open subsets of $V$.  Then $f=c_1\chi_{V_1}+\cdots+c_r\chi_{V_r}$ and so it suffices to show that if $U$ is a compact open subset of $\mathscr G$, then $\chi_U\in A$.

Since $\mathscr G^a$ is a basis for the topology of $\mathscr G$ and $U$ is compact open, it follows $U=U_1\cup \cdots \cup U_r$ with the $U_i\in \mathscr G^a$.  Since $U_i\subseteq U$, for $i=1,\ldots,r$, and $U$ is Hausdorff, it follows that any finite intersection of elements of the set $\{U_1,\ldots, U_r\}$  belongs to $\mathscr G^a$.
The principle of inclusion-exclusion yields:
\begin{equation}\label{charfuncs}
\chi_U=\chi_{U_1\cup\cdots\cup U_n} = \sum_{k=1}^n (-1)^{k-1}\sum_{\substack{I\subseteq \{1,\ldots,n\}\\ |I|=k}}\chi_{\bigcap_{i\in I}U_i}
\end{equation}
Hence $\chi_U\in A$, as required.
\end{proof}

We now define the convolution product on $K\mathscr G$ in order to make it a $K$-algebra.

\begin{Def}[Convolution]
Let $f,g\in K\mathscr G$.  Then their \emph{convolution} $f\ast g$ is defined, for $x\in \mathscr G$, by \[f\ast g(x) = \sum_{y\in d^{-1}d(x)}f(xy\inv)g(y).\]
\end{Def}

Of course, one must show that this sum is really finite and $f\ast g$ belongs to $K\mathscr G$, which is the content of the following proposition.

\begin{Prop}\label{convolutionwelldefined}
Let $f,g\in K\mathscr G$. Then:
\begin{enumerate}
\item $f\ast g\in K\mathscr G$;
\item If $f,g$ are continuous with compact support on $U,V\in \mathscr G^a$, respectively, then $f\ast g$ is continuous with compact support on $UV$;
\item If $U,V\in \mathscr G^a$, then $\chi_U\ast \chi_V=\chi_{UV}$;
\item If $U\in \mathscr G^a$, then $\chi_{U\inv}(x) = \chi_U(x\inv)$.
\end{enumerate}
\end{Prop}
\begin{proof}
Since the characteristic functions of elements of $\mathscr G^a$ span $K\mathscr G$ by Proposition~\ref{characteristicbasis}, it is easy to see that (1) and (2) are consequences of (3).  We proceed to the task at hand: establishing (3).

Indeed, we have
\begin{equation}\label{convolutionofslice}
\chi_U\ast \chi_V(x) = \sum_{y\in d\inv d(x)}\chi_U(xy\inv)\chi_V(y).
\end{equation}
Suppose first $x\in UV$.  Then we can find $a\in U$ and $b\in V$ so that $x=ab$.  Therefore, $a=xb\inv$, $d(x)=d(b)$ and $\chi_U(xb\inv)\chi_V(b)=1$.  Moreover, since $U$ and $V$ are slices, $b$ is the unique element of $V$ with $d(x)=d(b)$. Thus the right hand side of \eqref{convolutionofslice} is $1$.

Conversely, suppose $x\notin UV$ and let $y\in d\inv d(x)$.  If $y\notin V$, then $\chi_V(y)=0$.  On the other hand, if $y\in V$, then $xy\inv\notin U$, for otherwise we would have $x=xy\inv\cdot y\in UV$.  Thus $\chi_U(xy\inv)=0$.  Therefore, each term of the right hand side of \eqref{convolutionofslice} is zero and so $\chi_U\ast \chi_V = \chi_{UV}$, as required.

Statement (4) is trivial.
\end{proof}

The associativity of convolution is a straightforward, but tedious exercise~\cite{Exel,Paterson}.

\begin{Prop}
Let $K$ be a commutative ring with unit and $\mathscr G$ an ample groupoid.  Then $K\mathscr G$ equipped with convolution is a $K$-algebra.
\end{Prop}

If $K=\mathbb C$, we make $\CC \mathscr G$ into a $\ast$-algebra by defining $f^*(x) = \ov{f(x\inv)}$.

\begin{Cor}
The map $\p\colon \mathscr G^a\to K\mathscr G$ given by $\p(U) = \chi_U$ is a homomorphism.
\end{Cor}

\begin{Rmk}[Groups]
If $\mathscr G^0$ is a singleton, so that $\mathscr G$ is a discrete group, then $K\mathscr G$ is the usual group algebra.
\end{Rmk}

\begin{Rmk}[Locally compact boolean spaces]
In the case $\mathscr G=\mathscr G^0$, one has that $K\mathscr G$ is the subalgebra of $K^{\mathscr G}$ spanned by the characteristic functions of compact open subsets of $\mathscr G$ equipped with the pointwise product.  If $K=\mathbb F_2$, then $K\mathscr G\cong B(\mathscr G^0)$ viewed as a boolean ring.
\end{Rmk}

\begin{Rmk}[Discrete groupoids]\label{discretegroupoid}
Notice that if $\mathscr G$ is a discrete groupoid and $g\in \mathscr G$, then $\{g\}\in \mathscr G^a$ and $\delta_g=\chi_{\{g\}}$.   It follows easily that \[\delta_g\ast \delta_h= \begin{cases} \delta_{gh} & d(g)=r(h)\\ 0 &\text{else}.\end{cases}\]  Thus $K\mathscr G$ can be identified with the $K$-algebra having basis $\mathscr G$ and whose product extends that of $\mathscr G$ where we interpret undefined products as $0$.  This is exactly the groupoid algebra considered, for example, in~\cite{mobius1,mobius2}.
\end{Rmk}

Propositions~\ref{characteristicbasis} and~\ref{convolutionwelldefined} imply that $K\mathscr G$ is a quotient of the semigroup algebra $K\mathscr G^a$.  Clearly $K\mathscr G$ satisfies the relations $\chi_{U\cup V}=\chi_U+\chi_V$ whenever $U,V\in B(\mathscr G^0)$ with $U\cap V=\emptyset$.  We show that these relations define $K\mathscr G$ as a quotient of $K\mathscr G^a$ in the case that $\mathscr G$ is Hausdorff.  This result should hold in general (in analogy with the analytic setting~\cite{Paterson}), but so far we have been unsuccessful in proving it. First we need a definition.  If $U,V\in \mathscr G^a$, we say $U$ is \emph{orthogonal} to $V$, written $U\perp V$, if $UV\inv =\emptyset=U\inv V$.  In this case $U\cup V$ is a disjoint union and belongs to $\mathscr G^a$.  Indeed, $U\inv UV\inv V=\emptyset=UU\inv VV\inv$ and so $U,V$ have disjoint images under both $d$ and $r$.  Thus $d,r$ restrict to  homeomorphisms of $U\cup V$ with $U\inv U\cup V\inv V$ and $UU\inv\cup VV\inv$, respectively.  Consequently, $U\cup V$ is a compact open slice.   More generally, if $U_1,\ldots, U_n\in \mathscr G^a$ are pairwise orthogonal, then the union $U_1\cup\cdots \cup U_n$ is disjoint and a compact open slice. See~\cite{Paterson} for details.

\begin{Thm}\label{presentation}
Let $\mathscr G$ be a Hausdorff ample groupoid.  Then $K\mathscr G=K\mathscr G^a/I$ where $I$ is the ideal generated by all elements $U+V-(U\cup V)$ where $U,V$ are disjoint elements of $B(\mathscr G^0)$.
\end{Thm}
\begin{proof}
First of all, Propositions~\ref{characteristicbasis} and~\ref{convolutionwelldefined} yield a surjective homomorphism $\lambda\colon K\mathscr G^a\to K\mathscr G$ given by $U\mapsto \chi_U$ and evidentally $I\subseteq \ker \lambda$.  We establish the converse via several intermediate steps.  First note that $\emptyset\in I$ since $\emptyset = \emptyset+\emptyset -(\emptyset\cup \emptyset)\in I$.

\begin{Step}\label{additive}
Suppose that $U\perp V$ with $U,V\in \mathscr G^a$.  Then $U+V-(U\cup V)\in I$.
\end{Step}
\begin{proof}
Note $U\cup V = (U\cup V)(U\inv U\cup V\inv V)$ and $U\inv U\cap V\inv V=\emptyset$.  Thus \[U+V-(U\cup V) = (U\cup V)[U\inv U+V\inv V-(U\inv U\cup V\inv V)]\in I\] as required.
\end{proof}

The next step uses that $\mathscr G^a$ is closed under pairwise intersection and relative complement in the Hausdorff setting (Proposition~\ref{semibooleanalgebra}).

\begin{Step}\label{writeasunion}
 If $U_1,\ldots, U_n\in \mathscr G^a$, then we can find $V_1,\ldots, V_m\in \mathscr G^a$ so that:
\begin{enumerate}
\item $V_i\cap V_j=\emptyset$ all $i\neq j$;
\item $V_1\cup\cdots\cup V_m = U_1\cup\cdots\cup U_n$;
\item For all $1\leq i\leq m$ and $1\leq j\leq n$, either $V_i\subseteq U_j$ or $V_i\cap U_j=\emptyset$.
\end{enumerate}
\end{Step}
\begin{proof}
We induct on $n$, the case $n=1$ being trivial as we can take $m=1$ and $V_1=U_1$.  Assume the statement for $n-1$ and find pairwise disjoint elements $W_1,\ldots, W_m\in \mathscr G^a$ so that $W_1\cup\cdots \cup W_m=U_1\cup\cdots\cup U_{n-1}$ and, for all $1\leq i\leq m$ and $1\leq j\leq n-1$, either $W_i\subseteq U_j$ or $W_i\cap U_j=\emptyset$.  Set $V_i=W_i\cap U_n$, $V_i'=W_i\setminus U_n$ (for $i=1,\ldots,m$) and put $V_{m+1} = (U_n\setminus W_1)\cap \cdots \cap (U_n\setminus W_m)= U_n\setminus (W_1\cup\cdots\cup W_m)$.  All these sets are elements of $\mathscr G^a$,  some of which may be empty.

It is clear from the construction that the $V_i$ ($1\leq i\leq m+1$) and $V_i'$ ($1\leq i\leq m$) form a collection of pairwise disjoint subsets.  Moreover,
\begin{align*}
V_1\cup V_1'\cup\cdots\cup  V_m\cup V'_m\cup V_{m+1} &= W_1\cup\cdots\cup W_m\cup U_n\setminus (W_1\cup\cdots\cup W_m)\\
 &= U_1\cup\cdots \cup U_n
\end{align*}
and so the second condition holds.

For the final condition, first note that $V_{m+1}\subseteq U_n$ and intersects no other $U_i$.  On the other hand, for $1\leq i\leq m$, we have $V_i\subseteq U_n$ and $V_i'\cap U_n=\emptyset$.  For any $1\leq j\leq n-1$ and $1\leq i\leq m$, as $V_i,V_i'\subseteq W_i$, we have that $V_i\cap U_j\neq \emptyset$ implies $V_i\subseteq W_i\subseteq U_j$ and similarly $V_i'\cap U_j\neq \emptyset$ implies $V_i'\subseteq W_i\subseteq U_j$.  This establishes Step~\ref{writeasunion}.
\end{proof}

Our next step is an easy observation.
\begin{Step}\label{disjointgivesperp}
Suppose that $U,V\in \mathscr G^a$ are disjoint and have a common upper bound $W\in \mathscr G^a$. Then $U\perp V$.
\end{Step}
\begin{proof}
Suppose that $UV\inv \neq \emptyset$.  Then there exists $u\in U$ and $v\in V$ with $d(u)=d(v)$.  But $u,v\in W$ then implies $u=v$.  This contradicts that $U\cap V=\emptyset$.  The proof that $U\inv V=\emptyset$ is dual.
\end{proof}

We may now complete the proof.  Suppose $0=\sum_{i=1}^n c_i\chi_{U_i}$ with $c_i\in K$ and $U_i\in \mathscr G^a$, for $1\leq i\leq n$.  Choose $V_1,\ldots, V_m\in \mathscr G^a$, as per Step~\ref{writeasunion}.  Then, for each $j$, we can write $U_j = V_{i_1}\cup \cdots\cup V_{i_{k_j}}$ for certain indices $i_1,\ldots,i_{k_j}$.  Since the $V_{i_r}$ are pairwise disjoint subsets of the slice $U_j$, they are in fact mutually orthogonal by Step~\ref{disjointgivesperp}.  Repeated application of Step~\ref{additive} now yields that $U_j +I= V_{i_1}+\cdots+V_{i_{k_j}}+I$.  On the other hand, since the union is disjoint, clearly  $\chi_{U_j}= \chi_{V_{i_1}}+\cdots+\chi_{V_{i_{k_j}}}$.  We conclude that there exist  $d_j\in K$, for $j=1,\ldots, m$, so that $\sum_{i=1}^n c_iU_i+I = \sum_{i=1}^m d_jV_j+I$ and $0=\sum_{i=1}^n c_i\chi_{U_i}=\sum_{i=1}^m d_j\chi_{V_j}$.  But since the $V_j$ are disjoint, this immediately yields $d_1=\cdots=d_m=0$ and so $\sum_{i=1}^n c_iU_i\in I$, as required.
\end{proof}

Our next goal is to show that $K\mathscr G$ is unital if and only if $\mathscr G^0$ is compact.
\begin{Prop}\label{unital}
The $K$-algebra $K\mathscr G$ is unital if and only if $\mathscr G^0$ is compact.
\end{Prop}
\begin{proof}
 Suppose first $\mathscr G^0$ is compact.  Since it is open in the relative topology by Proposition~\ref{unitsareopen}, it follows that $u=\chi_{\mathscr G^0}\in K\mathscr G$.  Now if $f\in K\mathscr G$, then we compute
\[f\ast u(x) = \sum_{y\in d\inv d(x)}f(xy\inv)u(y) = f(x)\] since $d(x)$ is the unique element of $\mathscr G^0$ in $d\inv d(x)$.  Similarly, \[u\ast f(x) =  \sum_{y\in d\inv d(x)}u(xy\inv)f(y) = f(x)\] since $xy\inv \in\mathscr G^0$ implies $x=y$.  Thus $u$ is the multiplicative identity of $\mathscr G$.

Conversely, suppose $u$ is the multiplicative identity.  We first claim that $u=\chi_{\mathscr G^0}$.  Let $x\in \mathscr G$.     Choose a compact open set $U\subseteq \mathscr G^0$ with $d(x)\in U$. Suppose  first $x\notin \mathscr G^0$. Then \[0=\chi_U(x) = u\ast \chi_U(x) = \sum_{y\in d\inv d(x)}u(xy\inv)\chi_U(y) = u(x)\] since $\{d(x)\} = U\cap d\inv d(x)$.  Similarly, if $x\in \mathscr G^0$, then  we have \[1=\chi_U(x) = u\ast \chi_U(x) = \sum_{y\in d\inv d(x)}u(xy\inv)\chi_U(y) = u(x).\]  So we must show that $\chi_{\mathscr G^0}\in K\mathscr G$ implies that $\mathscr G^0$ is compact.

By Proposition~\ref{characteristicbasis}, there exist $U_1,\ldots, U_k\in \mathscr G^a$ and $c_1,\ldots, c_k\in K$ so that $\chi_{\mathscr G^0} = c_1\chi_{U_1}+\cdots+c_k\chi_{U_k}$.  Thus $\mathscr G^0\subseteq U_1\cup \cdots\cup  U_k$.  But then $\mathscr G^0= d(U_1)\cup \cdots\cup d(U_k)$.  But each $d(U_i)$ is compact, being homeomorphic to $U_i$, so $\mathscr G^0$ is compact, as required.
\end{proof}

The center of $K\mathscr G$ can be described by functions that are constant on conjugacy classes, analogously to the case of groups.

\begin{Def}[Class function]
Define $f\in K\mathscr G$ to be a \emph{class function} if:
\begin{enumerate}
\item $f(x)\neq 0$ implies $d(x)=r(x)$;
\item $d(x)=r(x)=d(z)$ implies $f(zxz\inv)=f(x)$.
\end{enumerate}
\end{Def}

\begin{Prop}
The center of $K\mathscr G$ is the set of class functions.
\end{Prop}
\begin{proof}
Suppose first that $f$ is a class function and $g\in K\mathscr G$.    Then
\begin{equation}\label{findcenter}
f\ast g(x) = \sum_{y\in d^{-1}d(x)}f(xy\inv)g(y) = \sum_{y\in d^{-1}d(x)\cap r^{-1}r(x)}f(xy\inv)g(y)
\end{equation}
since $f(xy\inv)=0$ if $r(x)=r(xy\inv)\neq d(xy\inv) = r(y)$.  But $f(xy\inv) = f(y(y\inv x)y\inv))=f(y\inv x)$ since $f$ is a class function and $d(y\inv x)=d(x)=d(y)=r(y\inv x)$.  Peforming the change of variables $z=y\inv x$, we obtain that the right hand side of \eqref{findcenter} is equal to \[\sum_{z\in d^{-1}d(x)\cap r^{-1}d(x)} g(xz\inv)f(z)= \sum_{z\in d^{-1}d(x)} g(xz\inv)f(z) = g\ast f(x)\] where the first equality uses that $f(z)=0$ if $d(z)\neq r(z)$.  Thus $f\in Z(K\mathscr G)$.

Conversely, suppose $f\in Z(K\mathscr G)$.  First we consider the case $x\in \mathscr G$ and $d(x)\neq r(x)$.  Choose a compact open set $U\subseteq \mathscr G^0$ so that $d(x)\in U$ and $r(x)\notin U$.  Then \[\chi_U\ast f(x) = \sum_{y\in d\inv d(x)} \chi_U(xy\inv)f(y) = 0\] since $xy\inv\in U$ forces it to be a unit, but then $y=x$ and $xx\inv =r(x)\notin U$.  On the other hand, \[f\ast \chi_U(x) = \sum_{y\in d\inv d(x)} f(xy\inv)\chi_U(y) = f(x)\] since $d(x)$ is the unique element of $d\inv d(x)$ in $U$.  Thus $f(x)=0$.

The remaining case is that  $d(x)=r(x)$ and we have $d(z)=d(x)$.  Then $zx\inv$ is defined.  Choose $U\in \mathscr G^a$ so that $zx\inv \in U$.  Then \[f\ast \chi_U(z) = \sum_{y\in d\inv d(z)}f(zy\inv)\chi_U(y) = f(zxz\inv)\] since $y\in U\cap d\inv d(z)$ implies $y=zx\inv$.  On the other hand, \[\chi_U\ast f(z) =   \sum_{y\in d\inv d(z)} \chi_U(zy\inv)f(y)=f(x)\] since $r(zy\inv) = r(zx\inv)$ and so $zy\inv \in U$ implies $zy\inv =zx\inv$, whence $y=x$. This shows that $f(x)=f(zxz\inv)$, completing the proof of the proposition.
\end{proof}

Our next proposition provides a sufficient condition for the characteristic functions of an inverse subsemigroup of $\mathscr G^a$ to span $K\mathscr G$.

\setcounter{Step}{0}

\begin{Prop}\label{smgroupgen}
Let $S\subseteq \mathscr G^a$ be an inverse subsemigroup such that:
\begin{enumerate}
 \item $E(S)$ generates the generalized boolean algebra $B(\mathscr G^0)$;
 \item $D=\{U\in \mathscr G^a\mid U\subseteq V\ \text{some}\ V\in S\}$ is a basis for the topology on $\mathscr G$.
\end{enumerate}
Then $K\mathscr G$ is spanned by the characteristic functions of elements of $S$.
\end{Prop}
\begin{proof}
Let $A$ be the span of the $\chi_V$ with $V\in S$.  Then $A$ is a $K$-subalgebra by Proposition~\ref{convolutionwelldefined}. We break the proof up into several steps.
\begin{Step}\label{step1}
 The collection $B$ of compact open subsets of $\mathscr G^0$ so that $\chi_U\in A$ is a generalized boolean algebra.
\end{Step}
 \begin{proof}
This is immediate from the formulas, for $U,V\in B(\mathscr G^0)$,
\begin{align*}
 \chi_U\ast \chi_V &= \chi_{UV} = \chi_{U\cap V}\\
\chi_{U\setminus V} & = \chi_U-\chi_{U\cap V}\\
\chi_{U\cup V} &= \chi_{U\setminus V}+\chi_{V\setminus U}+\chi_{U\cap V}.
\end{align*}
since $A$ is a subalgebra.
 \end{proof}
We may now conclude by (1) that $A$ contains $\chi_U$ for every element of $B(\mathscr G^0)$.

\begin{Step}\label{step2}
The characteristic function of each element of $D$ belongs to $A$.
\end{Step}
\begin{proof}
If $U\subseteq V$ with $V\in S$, then $VU\inv U=U$ and so $\chi_{U}=\chi_V\ast \chi_{U\inv U}\in A$ by Step~\ref{step1} since $U\inv U\in E(\mathscr G^a)=B(\mathscr G^0)$.
\end{proof}

\begin{Step}\label{step3}
Each $\chi_U$ with $U\in \mathscr G^a$ belongs to $A$.
\end{Step}
\begin{proof}
Since $D$ is a basis for $\mathscr G$ by hypothesis and $U$ is compact open, we may write $U=U_1\cup\cdots\cup U_n$ with the $U_i\in D$, for $i=1,\ldots, n$.  Since the $U_i$ are all contained in $U$, any finite intersection of the $U_i$ is clopen in $U$ and hence belongs to $\mathscr G^a$.  As $D$ is a downset, in fact, any finite intersection of the $U_i$ belongs to $D$.  Therefore, $\chi_U\in A$ by \eqref{charfuncs}.
\end{proof}

Proposition~\ref{characteristicbasis} now yields the result.
\end{proof}

\section{Actions of inverse semigroups and groupoids of germs}
As inverse semigroups are models of partial symmetry~\cite{Lawson}, it is natural to study them via their actions on spaces.  From such ``dynamical systems'' we can form a groupoid of germs and hence, in the ample setting, a $K$-algebra.  This $K$-algebra will be shown to behave like a cross product.

\subsection{The category of actions}
Let $X$ be a locally compact Hausdorff space and denote by $I_X$ the inverse semigroup of all homeomorphisms between open subsets of $X$.

\begin{Def}[Action]
An \emph{action} of an inverse semigroup $S$ on $X$ is a homomorphism $\p\colon S\to I_X$ written $s\mapsto \p_s$.  As usual, if $s\in S$ and $x\in \mathrm{dom}(\p_s)$, then we put $sx=\p_s(x)$.  Let us set $X_e$ to be the domain of $\p_e$ for  $e\in E(S)$.  The action is said to be \emph{non-degenerate} if $X=\bigcup_{e\in E(S)}X_e$.  If $\psi\colon S\to I_Y$ is another action, then a morphism from $\p$ to $\psi$ is a continuous map $\alpha\colon X\to Y$ so that, for all $x\in X$, one has that $sx$ is defined if and only if $s\alpha(x)$ is defined, in which case $\alpha(sx) = s\alpha(x)$.
\end{Def}

We will be most interested in what we term ``ample'' actions.  These actions will give rise to ample groupoids via the groupoid of germs constructions.

\begin{Def}[Ample action]
A non-degenerate action $\p\colon S\to I_X$ of an inverse semigroup $S$ on a space $X$ is said to be \emph{ample} if:
\begin{enumerate}
 \item $X$ is a locally compact boolean space;
\item $X_e\in B(X)$, for all $e\in E(S)$.
\end{enumerate}
If in addition, the collection $\{X_e\mid e\in E(S)\}$ generates $B(X)$ as a generalized boolean algebra, we say the action is \emph{boolean}.
\end{Def}

If $\mathscr G$ is an ample groupoid, there is a natural boolean action of $\mathscr G^a$ on $\mathscr G^0$.  Namely, if $U\in \mathscr G^a$, then the domain of its action is $U\inv U$ and the range is $UU\inv$.  If $x\in U\inv U$, then there is a unique element $g\in U$ with $d(g)=x$.  Define $Ux=r(g)$.  This is exactly the partial homeomorphism whose ``graph'' is $U$.  See~\cite{Exel,Paterson} for details.

\begin{Prop}\label{propermap}
Suppose that $S$ has ample actions on $X$ and $Y$ and let $\alpha\colon X\to Y$ be a morphism of the actions.  Then:
\begin{enumerate}
 \item For each $e\in E(S)$, one has $\alpha\inv(Y_e)=X_e$;
\item $\alpha$ is proper;
\item $\alpha$ is closed.
\end{enumerate}
\end{Prop}
\begin{proof}
From the definition of a morphism, $\alpha(x)\in Y_e$ if and only if $e\alpha(x)$ is defined, if and only if $ex$ is defined, if and only if $x\in X_e$.  This proves (1).  For (2), let $C\subseteq Y$ be compact.  Since the action on $Y$ is non-degenerate, we have that $C\subseteq \bigcup_{e\in E(S)} Y_e$ and so by compactness of $C$ it follows that $C\subseteq Y_{e_1}\cup\cdots\cup Y_{e_n}$ for some idempotents $e_1,\ldots,e_n$.  Thus $\alpha\inv (C)$ is a closed subspace of $\alpha\inv (Y_{e_1}\cup\cdots\cup Y_{e_n}) = X_{e_1}\cup\cdots \cup X_{e_n}$ and hence is compact since the $X_{e_i}$ are compact.  Finally, it is well known that a proper map between locally compact Hausdorff spaces is closed.
\end{proof}

For us the main example of a boolean action is the action of an inverse semigroup $S$ on the spectrum of its semilattice of idempotents.  For a semilattice $E$, we denote by $\wh E$ the space of non-zero semilattice homomorphisms $\p\colon E\to \{0,1\}$ topologized as a subspace of $\{0,1\}^E$. Such a homomorphism extends uniquely to a boolean ring homomorphism $\mathbb F_2E\to \mathbb F_2$ and hence $\wh E\cong \mathrm{Spec}(\mathbb F_2E)$, and is in particular a locally compact boolean space with $B(\wh E)\cong \mathbb F_2E$ (viewed as a generalized boolean algebra).  For $e\in E$, define $D(e) = \{\p\in \wh E\mid \p(e)=1\}$;  this is the compact open set corresponding to $e$ under the above isomorphism.  The semilattice of subsets of the form $D(e)$ generates $B(\wh E)$ as a generalized boolean algebra because $E$ generates $\mathbb F_2E$ as a boolean ring.  In fact, the map $e\mapsto D(e)$ is the universal semilattice homomorphism of $E$ into a generalized boolean algebra (corresponding to the universal property of the inclusion $E\to \mathbb F_2E$).  Elements of $\wh E$ are often referred to as \emph{characters}.

\begin{Def}[Spectral action]
Suppose that $S$ is an inverse semigroup with semilattice of idempotents $E(S)$.  To each $s\in S$, there is an associated homeomorphism $\beta_s\colon D(s^*s)\to D(ss^*)$ given by $\beta_s(\p)(e) = \p(s^*es)$.  The map $s\mapsto \beta_s$ provides a boolean action $\beta\colon S\to I_{\wh{E(S)}}$~\cite{Exel,Paterson}, which we call the \emph{spectral action}.
\end{Def}

The spectral action enjoys the following universal property.

\begin{Prop}\label{universal}
Let $\mathscr C$ be the category of boolean actions of $S$.
\begin{enumerate}
\item Suppose $S$ has a boolean action on $X$ and an ample action on $Y$ and let $\psi\colon X\to Y$ be a morphism.  Then $\psi$ is a topological embedding of $X$ as a closed subspace of $Y$.
\item Each homset of $\mathscr C$ contains at most one element.
\item The spectral action $\beta\colon S\to I_{\wh{E(S)}}$ is the terminal object in $\mathscr C$.
\end{enumerate}
\end{Prop}
\begin{proof}
By Proposition~\ref{propermap}, the map $\psi$ is proper  and $\psi\inv\colon B(Y)\to B(X)$ takes $Y_e$ to $X_e$ for $e\in E(S)$.   Since the $X_e$ with $e\in E(S)$ generate $B(X)$ as a generalized boolean algebra, it follows that $\psi\inv \colon B(Y)\to B(X)$ is surjective.  By Stone duality, we conclude $\psi$ is injective.  Also $\psi$ is closed being proper.  This establishes (1).

Again by Proposition~\ref{propermap} if $\psi\colon X\to Y$ is a morphism of boolean actions, then $\psi$ is proper and $\psi\inv\colon B(Y)\to B(X)$ sends $Y_e$ to $X_e$.  Since the $Y_e$ with $e\in E(S)$ generate $B(Y)$, it follows $\psi\inv$ is uniquely determined by the actions of $S$ on $X$ and $Y$.  But $\psi\inv$ determines $\psi$ by Stone duality, yielding (2).

Let $\alpha\colon S\to I_Y$ be a boolean action.  By definition the map $e\mapsto Y_e$ yields a homomorphism $E(S)\to B(Y)$ extending to a surjective homomorphism $\mathbb F_2E(S)\to B(Y)$.  Recalling $\mathbb F_2E(S)\cong B(\wh {E(S)})$, Stone duality yields a proper continuous injective map $\psi\colon Y\to \wh{E(S)}$ that sends $y\in Y$ to $\p_y\colon E(S)\to \{0,1\}$ given by $\p_y(e) = \chi_{Y_e}(y)$.  It remains to show that the map $y\mapsto \p_y$ is a morphism.  Let $y\in Y$ and $s\in S$.  Then $sy$ is defined if and only if $y\in Y_{s^*s}$, if and only if $\p_y(s^*s)=1$, if and only if $\p_y\in D(s^*s)$.  If $y\in Y_{s^*s}$, then $s\p_y(e) = \p_y(s^*es) = \chi_{Y_{s^*es}}(y)$.  But $Y_{s^*es}$ is the domain of $\alpha_{es}$.  Since $sy$ is defined, $y\in Y_{s^*es}$ if and only if $sy\in Y_e$.  Thus $\chi_{Y_{s^*es}}(y) = \chi_{Y_e}(sy) = \p_{sy}(e)$.  We conclude that $\p_{sy} = s\p_y$.  This completes the proof of the theorem.
\end{proof}

Since the restriction of a boolean action to a closed invariant subspace is evidentally boolean, we obtain the following description of boolean actions, which was proved by Paterson in a slightly different language~\cite{Paterson}.

\begin{Cor}
There is an equivalence between the category of boolean actions of $S$ and the poset of $S$-invariant closed subspaces of $\wh{E(S)}$.
\end{Cor}

%The following easy proposition will be used later.
%
%\begin{Prop}\label{pullbackspectra}
%Let $\p\colon S\to T$ be a surjective homomorphism of inverse semigroups.  Then the composition of $\p$ with the spectral action of $T$ gives a boolean action of $S$ on $\wh{E(T)}$.  Consequently, $\wh{E(T)}$ can be identified with a closed invariant subspace of $\wh{E(S)}$ by sending $\chi\in \wh{E(T)}$ to $\chi\p\in \wh{E(S)}$.
%\end{Prop}
%\begin{proof}
%Since $\p$ takes $E(S)$ onto $E(T)$~\cite{Lawson}, it is immediate that the action is boolean.  The final statement follows from Proposition~\ref{universal}.
%\end{proof}

\subsubsection{Filters}
Recall that a \emph{filter} $\mathscr F$ on a semilattice $E$ is a non-empty subset that is closed under pairwise meets and is an upset in the ordering.  For example, if $\p\colon E\to \{0,1\}$ is a character, then $\p\inv(1)$ is a filter.  Conversely, if $\mathscr F$ is a filter, then its characteristic function $\chi_{\mathscr F}$ is a non-zero homomorphism.  A filter is called \emph{principal} if it has a minimum element, i.e., is of the form $e^{\uparrow}$.  A character of the form $\chi_{e^{\uparrow}}$, with $e\in E$, is called a \emph{principal character}. Notice that every filter on a finite semilattice $E$ is principal and in this case $\wh E$ is homeomorphic to $E$ with the discrete topology.

In general, the set of principal characters is dense in $\wh{E}$ since $\chi_{e^{\uparrow}}\in D(e)$ and the generalized boolean algebra generated by the open subsets of the form $D(e)$ is a basis for the topology on $\wh{E}$.  Thus if $\wh{E}$ is discrete, then necessarily each filter on $E$ is principal and $\wh{E}$ is in bijection with $E$.  However, the converse is false, as we shall see in a moment.   The following, assumedly well-known,  proposition captures when every filter is principal, when the topology on $\wh E$ is discrete and when the principal characters are discrete in $\wh E$.

\begin{Prop}\label{descendingchain}
Let $E$ be a semilattice. Then:
\begin{enumerate}
\item Each filter on $E$ is principal if and only if $E$ satisfies the descending chain condition;
\item The topology on $\wh E$ is discrete if and only if each principal downset of $E$ is finite;
\item The set of principal characters is discrete in $\wh E$ if and only if, for all $e\in E$, the downset $e^\downharpoonleft = \{f\in E\mid f<e\}$ is finitely generated.
\end{enumerate}
\end{Prop}
\begin{proof}
Suppose first $E$ satisfies the descending chain condition and let $\mathscr F$ be a filter.  Then each element $e\in \mathscr F$ is above a minimal element of $\mathscr F$, else we could construct an infinite strictly descending chain.  But if $e,f\in \mathscr F$ are minimal, then $ef\in \mathscr F$ implies that $e=ef=f$.  Thus $\mathscr F$ is principal.  Conversely, suppose each filter in $E$ is principal and that $e_1\geq e_2\geq \cdots$ is a descending chain.  Let $\mathscr F = \bigcup_{i=1}^{\infty} e_i^{\uparrow}$.  Then $\mathscr F$ is a filter.  By assumption, we have $\mathscr F=e^{\uparrow}$ for some $e\in E$.  Because $e\in \mathscr F$, we must have $e\geq e_i$ for some $i$.  On the other hand $e_j\geq e$ for all $j$ since $\mathscr F=e^{\uparrow}$.  If follows that $e=e_i=e_{i+1}=\cdots$ and so $E$ satisfies the descending chain condition.

To prove (2), suppose first that $\wh E$ is discrete.  Then since $D(e)$ is compact open, it must be finite.  Moreover, each filter on $E$ is principal and  so $D(e) = \{\chi_{f^{\uparrow}}\mid f\leq e\}$.  It follows that $e^{\downarrow}$ is finite.  Conversely, if each principal downset is finite, then every filter on $E$ is principal by (1). Suppose $e^{\downarrow}\setminus \{e\} = \{f_1,\ldots,f_n\}$. Then $\{\chi_{e^{\uparrow}}\} = D(e)\setminus (D(f_1)\cup\cdots\cup D(f_n))$ is open.  Thus $\wh E$ is discrete.

For (3), suppose first $e^\downharpoonleft$ is generated by $f_1,\ldots,f_n$.  Then $\chi_{e^{\uparrow}}$ is the only principal character in the open set $D(e)\setminus (D(f_1)\cup \cdots \cup D(f_n))$.  This establishes sufficiency of the given condition for discreteness.  Conversely, suppose that the principal characters form a discrete set.  Then there is a basic neighborhood of $\chi_{e^{\uparrow}}$ of the form $U=D(e')\setminus (D(f_1)\cup \cdots \cup D(f_n))$ for certain $e',f_1,\ldots, f_n\in E$ (cf.~\cite{Paterson}) containing no other principal character.  Then $e\leq e'$ and $e\nleq f_i$, for $i=1,\ldots, n$. In particular, $ef_1,\ldots,ef_n\in  e^\downharpoonleft$.  We claim they generate it.   Indeed, if $f<e\leq e'$, then since $\chi_{f^\uparrow}\notin U$, we must have $f\leq f_i$ for some $i=1,\ldots, n$ and hence $f=ef\leq ef_i$, for some $i=1,\ldots n$.  This completes the proof.
\end{proof}

For instance, consider the semilattice $E$ with underlying set $\mathbb N\cup \{\infty\}$ and with order given by $0<i<\infty$ for all $i\geq 1$ and all other elements are incomparable.  Then $E$ satisfies the descending chain condition but $\infty^{\downarrow}$ is infinite.   If we identify $\wh E$ with $E$ as sets, then the topology is that of $\infty$ being a one-point compactification of the discrete space $\mathbb N$.  The condition in (3) is called \emph{pseudofiniteness} in~\cite{Munnsemiprimitive}.

%From now on we identify characters with filters when convenient.  In particular, if $\mathscr F$ is a filter on $E(S)$, for an inverse semigroup, then one checks that $s\mathscr F$ is defined if and only if $s^*s\in \mathscr F$, in which case the resulting filter is $s\mathscr F = \{e\mid s^*es\in \mathscr F\}$.

\begin{Def}[Ultrafilter]
A filter $\mathscr F$ on a semilattice $E$ is called an \emph{ultrafilter} if it is a maximal proper filter.
\end{Def}

Recall that an idempotent $e$ of an inverse semigroup is called \emph{primitive} if it is minimal amongst its non-zero idempotents.   The connection between ultrafilters and morphisms of generalized boolean algebras is well known~\cite{halmosnew}.

\begin{Prop}\label{ultrafilterchar}
Let $E$ be a semilattice with zero.
\begin{enumerate}
\item A principal filter $e^{\uparrow}$ is an ultrafilter if and only if $e$ is primitive.
\item Moreover, if $E$ is a generalized boolean algebra, then a filter $\mathscr F$ on $E$ is an ultrafilter if and only if $\chi_{\mathscr F}\colon E\to \{0,1\}$ is a morphism of generalized boolean algebras.
\end{enumerate}
\end{Prop}
\begin{proof}
Evidentally, if $e$ is not a minimal non-zero idempotent, then $e^{\uparrow}$ is not an ultrafilter since it is contained in some proper principal filter.  Suppose that $e$ is primitive and $f\notin e^{\uparrow}$.  Then $ef< e$ and so $ef=0$.  Thus no proper filter contains $e^{\uparrow}$. This proves (1).

To prove (2), suppose first that $\mathscr F$ is an ultrafilter.  We must verify that $e_1\vee e_2\in \mathscr F$ implies $e_i\in \mathscr F$ for some $i=1,2$.  Suppose neither belong to $\mathscr F$.  For $i=1,2$, put $\mathscr F_i = \{e\in E\mid \exists f\in \mathscr F\ \text{such that}\ e\geq e_if\}$.  Then $\mathscr F_i$, for $i=1,2$, are filters properly containing $\mathscr F$.  Thus $0\in \mathscr F_1\cap\mathscr F_2$ and so we can find $f_1,f_2\in F$ with $e_1f_1=0=e_2f_2$.  Then $f_1f_2\in \mathscr F$ and $0=e_1f_1f_2\vee e_2f_1f_2=(e_1\vee e_2)f_1f_2\in \mathscr F$, a contradiction.  Thus $e_i\in \mathscr F_i$ some $i=1,2$.

Conversely, suppose that $\chi_{\mathscr F}$ is a morphism of generalized boolean algebras.  Then $0\notin \mathscr F$ and so $\mathscr F$ is a proper filter.  Suppose that $\mathscr F'\supsetneq \mathscr F$ is a filter.  Let $e\in \mathscr F'\setminus \mathscr F$ and let $f\in \mathscr F$. We cannot have $fe\in \mathscr F$ as $f\notin \mathscr F$.  Because $fe\vee (f\setminus e)=f\in \mathscr F$ and $\chi_{\mathscr F}$ is a morphism of generalized boolean algebras, it follows that $f\setminus e\in \mathscr F\subseteq \mathscr F'$.  Thus $0=e(f\setminus e)\in \mathscr F'$.  We conclude $\mathscr F$ is an ultrafilter.
\end{proof}

It follows from the proposition that if $E$ is a generalized boolean algebra, then the points of $\mathrm{Spec}(E)$ can be identified with ultrafilters on $E$.  If $X$ is a locally compact boolean space and $x\in X$, then the corresponding ultrafilter on $B(X)$ is the set of all compact open neighborhoods of $x$.  It is not hard to see that $\mathrm{Spec}(E)$ is a closed subspace of $\wh E$ for a generalized boolean algebra $E$.  For semilattices in general, the space of ultrafilters is not closed in $\wh E$, which led Exel to consider the closure of the space of ultrafilters, which he terms the space of tight filters~\cite{Exel}.

%The space of ultrafilters on $E(S)$ is invariant under the spectral action.
%
%\begin{Prop}\label{ultrafiltersclosed}
%Let $S$ be an inverse semigroup and let $\mathscr U$ be the set of characters $\chi\in \wh {E(S)}$ such that $\chi\inv(1)$ is an ultrafitler.  Then $\mathscr U$ is invariant under the action of $S$.
%\end{Prop}
%\begin{proof}
%It is convenient in this proof to identify characters with filters, so we view $\mathscr U$ as the space of ultrafilters. Suppose $\mathscr F\in \mathscr U$ with $s^*s\in \mathscr F$.  Assume that $\mathscr F'$ is a filter properly containing $s\mathscr F$.  We must show that $\mathscr F'=E(S)$.  It follows from the definition that the action of $S$ is order preserving (and hence by order isomorphisms). Since $ss^*\in s\mathscr F\subseteq \mathscr F'$, it follows $s^*\mathscr F'$ is defined and $s^*\mathscr F'\supsetneq s^*s\mathscr F=\mathscr F$.  Thus $s^*\mathscr F'=E(S)$ and hence $\mathscr F'=ss^*\mathscr F' = E(S)$.  We conclude $s\mathscr F\in \mathscr U$.
%\end{proof}

\subsection{Groupoids of germs}
There is a well-known construction assigning to each non-degenerate action $\p\colon S\to I_X$ of an inverse semigroup $S$ on a locally compact Hausdorff space $X$ an \'etale groupoid, which we denote $S\ltimes_\p X$, known as the \emph{groupoid of germs} of the action~\cite{Paterson,Exel,resendeetale}.  Usually, $\p$ is dropped from the notation if it is understood.  The groupoid of germs construction is functorial.  It goes as follows.  As a set $S\ltimes_\p X$ is the quotient of the set $\{(s,x)\in S\times X\mid x\in X_{s^*s}\}$ by the equivalence relation that identifies $(s,x)$ and $(t,y)$ if and only if $x=y$ and there exists $u\leq s,t$ with $x\in X_{u^*u}$   One writes $[s,x]$ for the equivalence class of $(s,x)$ and calls it the \emph{germ of $s$ at $x$}.  The associated topology is the so-called germ topology.  A basis consists of all sets of the form $(s,U)$ where $U\subseteq X_{s^*s}$ and $(s,U) = \{[s,x]\mid x\in U\}$.  The multiplication is given by defining $[s,x]\cdot [t,y]$ if and only if $ty=x$, in which case the product is $[st,y]$.  The units are the elements $[e,x]$ with $e\in E(S)$ and $x\in X_e$.  The projection $[e,x]\mapsto x$ gives a homeomorphism of the unit space of $S\ltimes X$ with $X$, and so from now on we identify the unit space with $X$.  One then has $d([s,x]) =x$ and $r([s,x])=sx$.  The inversion is given by $[s,x]\inv = [s^*,sx]$.  The groupoid $S\ltimes X$ is an \'etale groupoid~\cite[Proposition 4.17]{Exel}.  The reader should consult~\cite{Exel,Paterson} for details.

Observe that if $\mathscr B$ is a basis for the topology of $X$, then a basis for $S\ltimes X$ consists of those sets of the form $(s,U)$ with $U\in \mathscr B$, as is easily verified.  Let us turn to some enlightening examples.

\begin{Example}[Transformation groupoids]
In the case that $S$ is a discrete group, the equivalence relation on $S\times X$ is trivial and the topology is the product topology.  The resulting \'etale groupoid is consequently Hausdorff and is known in the literature as the \emph{transformation groupoid} associated to the transformation group $(S,X)$~\cite{Renault,Paterson}.
\end{Example}

\begin{Example}[Maximal group image]
 An easy example is the case when $X$ is a one-point set on which $S$ acts trivially.  It is then straightforward to see that $S\ltimes X$ is the maximal group image of $S$.  Indeed, elements of the groupoid are equivalence classes of elements of $S$ where two elements are considered equivalent if they have a common lower bound.
\end{Example}

\begin{Example}[Underlying groupoid]
Another example is the case $X=E(S)$ with the discrete topology.  The action is the so-called Munn representation $\mu\colon S\to I_{E(S)}$ given by putting $X_e = e^{\downarrow}$ and defining \mbox{$\mu_s\colon X_{s^*s}\to X_{ss^*}$} by $\mu_s(e) = ses^*$~\cite{Lawson}.  The spectral action is the dual of the (right) Munn representation.  We observe that each equivalence class $[s,e]$ contains a unique element of the form $(t,e)$ with $t^*t=e$, namely $t=se$.  Then $e$ is determined by $t$.  Thus arrows of $S\ltimes X$ are in bijection with elements of $S$ via the map $s\mapsto [s,s^*s]$.  One has $d(s)=s^*s$, $r(s)=ss^*$ and if $s,t$ are composable, their composition is $st$.  The inverse of $s$ is $s^*$.  This is the so-called \emph{underlying groupoid} of $S$~\cite{Lawson}.  The slice $(s,\{s^*s\})$ contains just the arrow $s$, so the topology is discrete.
\end{Example}

The next proposition establishes the functoriality of the germ groupoid.

\begin{Prop}\label{functoriality}
Let $S$ act on $X$ and $Y$ and suppose $\psi\colon X\to Y$ is a morphism.  Then there is a continuous functor $\Psi\colon S\ltimes X\to S\ltimes Y$ given by $[s,x]\mapsto [s,\psi(x)]$.  If the actions are boolean, then $\Psi$ is an embedding of groupoids and the image consists precisely of those arrows of $S\ltimes Y$ between elements of $\psi(X)$.
\end{Prop}
\begin{proof}
We verify that $\Psi$ is well defined.   First note that $x\in X_{s^*s}$ if and only if $\psi(x)\in Y_{s^*s}$ for any $s\in S$ by the definition of a morphism.  Suppose that $(s,x)$ is equivalent to $(t,x)$.  Then we can find $u\leq s,t$ with $x\in X_{u^*u}$.  It then follows that $\psi(x)\in Y_{u^*u}$ and so $(s,\psi(x))$ is equivalent to $(t,\psi(x))$.  Thus $\Psi$ is well defined. The details that $\Psi$ is a continuous functor are routine and left to the reader.  In the case the actions are boolean, the fact that $\Psi$ is an embedding follows easily from Proposition~\ref{universal}.  The description of the image follows because if $[s,y]$ satisfies $y, sy\in \psi(X)$ and $y=\psi(x)$, then $[s,y]=\Psi([s,x])$.
\end{proof}

In~\cite{Exel,Paterson}, the term \emph{reduction} is used to describe a groupoid obtained by restricting the units to a closed subspace and taking the full subgroupoid (in the sense of~\cite{Mac-CWM}) of all arrows between these objects.

The following is~\cite[Proposition 4,18]{Exel}.

\begin{Prop}
The basic open set $(s,U)$ with $U\subseteq X_{s^*s}$ is a slice of $S\ltimes X$ homeomorphic to $U$.
\end{Prop}

From the proposition, it easily follows that if $\p\colon S\to I_X$ is an ample action and $\mathscr G=S\ltimes_{\p} X$, then each open set of the form $(s,U)$ with $U\in B(X)$ belongs to $\mathscr G^a$ and the collection of such open sets is a basis for the topology on $\mathscr G$.  Thus $S\ltimes_{\p} X$ is an ample groupoid. Moreover, the map $s\mapsto (s,X_{s^*s})$ in this setting is a homomorphism $\theta\colon S\to \mathscr G^a$, as the following lemma shows in the general case.

\begin{Lemma}\label{slicehom}
Let $S$ have a non-degenerate action on $X$. If $(s,U)$ and $(t,V)$ are basic neighborhoods of $\mathscr G=S\ltimes X$, then \[(s,U)(t,V) = (st, t^*(U\cap tV))\] and $(s,U)\inv = (s^*,sU)$.  Consequently, the map $s\mapsto (s,X_{s^*s})$ is a homomorphism from $S$ to $\mathscr G^{op}$ and furthermore if the action is ample, then it is a homomorphism to $\mathscr G^a$.
\end{Lemma}
\begin{proof}
First observe that $t^*(X_{s^*s}\cap tX_{t^*t}) = t^*(X_{s^*s}\cap X_{tt^*}) = t^*(X_{s^*stt^*}) = t^*s^*stt^*(X_{s^*stt^*}) = t^*s^*s(X_{s^*stt^*}) = X_{t^*s^*st} = X_{(st)^*(st)}$.  The final statement now follows from the first.

By definition $U\subseteq X_{s^*s}$ and $V\subseteq X_{t^*t}$. Therefore, we have $t^*(U\cap tV)\subseteq  t^*(X_{s^*s}\cap tX_{t^*t})= X_{(st)^*st}$ and so $(st, t^*(U\cap tV))$ is well defined.   Suppose $x\in U$ and $y\in V$ with $ty=x$.  Then $x\in U\cap tV$ and so $y\in t^*(U\cap tV)$.  Moreover, $[s,x]\cdot [t,y] = [st,y]\in (st,t^*(U\cap tV))$.  This shows $(s,U)(t,V)\subseteq (st,t^*(U\cap tV))$.  For the converse, assume $[st,y]\in (st,t^*(U\cap tV))$.  Then $y\in t^*tV=V$ and if we set $x=ty$, then $x\in tt^*U\subseteq U$.    We conclude $[st,y]=[s,x]\cdot[t,y]\in (s,U)(t,V)$.  The equality $(s,U)\inv = (s^*,sU)$ is trivial.
\end{proof}

Notice that the action of the slice $(s,X_{s^*s})$ on $\mathscr G^0=X$ is exactly the map $x\mapsto sx$.  Indeed, the domain of the action of $(s,X_{s^*s})$ is $X_{s^*s}$ and if $x\in X_{s^*s}$, then $[s,x]$ is the unique element of the slice with domain $x$.  But $r([s,x]) = sx$.

In the case of the trivial action on a one-point space, the map from Lemma~\ref{slicehom} is the maximal group image homomorphism.  In the case of the Munn representation, the map is $s\mapsto \{t\in S\mid t\leq s\}=s^{\downarrow}$.  It is straightforward to verify that in this case the homomorphism is injective.  The reader should compare with~\cite{mobius1,mobius2}.

Summarizing, we have the following proposition.

\begin{Prop}\label{amplestuff}
 Let $\p\colon S\to I_X$ be an ample action and put $\mathscr G=S\ltimes_{\p} X$.  Then:
\begin{enumerate}
 \item $\mathscr G$ is an ample groupoid;
\item There is a homomorphism $\theta\colon S\to \mathscr G^a$ given by $\theta(s) = (s,X_{s^*s})$;
\item $\bigcup \theta(S) = \mathscr G$;
\item $\{U\in \mathscr G^a\mid U\subseteq \theta(s)\ \text{some}\ s\in S\}$ is a basis for the topology on $\mathscr G$;
\item There is a homomorphism $\Theta\colon S\to K\mathscr G$ given by \[\Theta(s) = \chi_{\theta(s)} = \chi_{(s,X_{s^*s})},\] which is a $\ast$-homomorphism when $K=\mathbb C$.
\end{enumerate}
Moreover, if $\p$ is a boolean action, then:
\begin{enumerate}
\item [(6)] $E(\theta(S))$ generates $B(\mathscr G^0)\cong B(X)$ as a generalized boolean algebra;
\item [(7)] $\Theta(S)$ spans $K\mathscr G$.
\end{enumerate}
\end{Prop}
\begin{proof}
Item (4) is a consequence of the fact that if $U\subseteq X_{s^*s}$, then the basic set $(s,U)$ is contained in $\theta(s)=(s,X_{s^*s})$.
Item (5) follows from Proposition~\ref{convolutionwelldefined} and Lemma~\ref{slicehom}.
Item (7) is a consequence of (4), (6) and Proposition~\ref{smgroupgen}.  The remaining statements are clear.
\end{proof}

The universal groupoid of an inverse semigroup was introduced by Paterson~\cite{Paterson} and has since been studied by several authors~\cite{Exel,lenz,resendeetale,strongmorita}.

\begin{Def}[Universal groupoid]
Let $S$ be an inverse semigroup.  The groupoid of germs $\mathscr G(S)=S\ltimes_\beta \wh{E(S)}$ is called the \emph{universal groupoid} of $S$~\cite{Paterson,Exel}.  It is an ample groupoid. Notice that if $E(S)$ is finite (or more generally if each principal downset of $E(S)$ is finite), then the underlying groupoid of $S$ is the universal groupoid.
\end{Def}

Propositions~\ref{universal} and~\ref{functoriality} immediately imply the following universal property of $\mathscr G(S)$, due to Paterson~\cite{Paterson}.

\begin{Prop}
Let $\p\colon S\to I_X$ be a boolean action.  Then there is a unique continuous functor $\Phi\colon S\ltimes X\to \mathscr G(S)$ so that $\Phi((s,X_{s^*s})) = (s,D(s^*s))$.  Moreover, $\Phi$ is an embedding of topological groupoids with image a reduction of $\mathscr G(S)$ to a closed $S$-invariant subspace of $\wh{E(S)}$.
\end{Prop}

Examples of universal groupoids of inverse semigroups can be found in~\cite[Chapter 4]{Paterson}.
It is convenient at times to use the following algebraic embedding of the underlying groupoid into the universal groupoid.

\begin{Lemma}\label{embedunderlying}
Let $s\in S$.  Then $[s,\chi_{(s^*s)^{\uparrow}}]=[t,\chi_{(s^*s)^{\uparrow}}]$ if and only if $s\leq t$.  Consequently, $[s,\chi_{(s^*s)^{\uparrow}}]\in (t,D(t^*t))$ if and only if $s\leq t$.   Moreover, the map $s\mapsto [s,\chi_{(s^*s)^{\uparrow}}]$ is an injective functor from the underlying groupoid of $S$ onto a dense subgroupoid of $\mathscr G(S)$.
\end{Lemma}
\begin{proof}
We verify the first two statements.  The final statement is straightforward from the previous ones and can be found in~\cite[Proposition~4.4.6]{Paterson}.   If $s\leq t$, the germs of $s$ and $t$ at $\chi_{(s^*s)^{\uparrow}}$ clearly coincide.  Assume conversely, that the germs are the same. By definition there exists $u\leq s,t$ so that $\chi_{(s^*s)^{\uparrow}}\in D(u^*u)$, i.e., $u^*u\geq s^*s$.  But then $u=su^*u= ss^*su^*u=ss^*s=s$.  Thus $s\leq t$.  The second statement follows since $[s,\chi_{(s^*s)^{\uparrow}}]\in (t,D(t^*t))$ if and only if $[s,\chi_{(s^*s)^{\uparrow}}]=[t,\chi_{(s^*s)^{\uparrow}}]$.
\end{proof}

We end this section with a sufficient condition for the groupoid of germs of an action to be Hausdorff, improving on~\cite[Proposition 6.2]{Exel} (see also~\cite{Paterson,lenz}).  For the universal groupoid, the condition is shown also to be necessary and is the converse to~\cite[Corollary 4.3.1]{Paterson}.  As a consequence we obtain a much easier proof that the universal groupoid of a certain commutative inverse semigroup considered in~\cite[Appendix C]{Paterson} is not Hausdorff.

Observe that a poset $P$ is a semilattice if and only if the intersection of principal downsets is again a principal downset.  We say that a poset is a \emph{weak semilattice} if the intersection of principal downsets is finitely generated as a downset. The empty downset is considered to be generated by the empty set.

\begin{Thm}\label{Hausdorff}
Let $S$ be an inverse semigroup.  Then the following are equivalent:
\begin{enumerate}
\item $S$ is a weak semilattice;
\item The groupoid of germs of any non-degenerate action \mbox{$\theta\colon S\to I_X$} such that $X_e$ is clopen for all $e\in E(S)$ is Hausdorff;
\item $\mathscr G(S)$ is Hausdorff.
\end{enumerate}
In particular, every groupoid of germs for an ample action of $S$ is Hausdorff if and only if $S$ is a weak semilattice.
\end{Thm}
\begin{proof}
First we show that (1) implies (2).
Suppose $[s,x]\neq [t,y]$ are elements of $\mathscr G$.  If $x\neq y$, then choose disjoint neighborhoods $U,V$ of $x$ and $y$ in $X$, respectively.  Clearly, $(s,U\cap X_{s^*s})$ and $(t,V\cap X_{t^*t})$ are disjoint neighborhoods of $[s,x]$ and $[t,y]$, respectively.

Next assume $x=y$.  If $s^{\downarrow}\cap t^{\downarrow}=\emptyset$, then $(s,X_{s^*s})$ and $(t,X_{t^*t})$ are disjoint neighborhoods of $[s,x]$ and $[t,x]$, respectively, since if $[s,z]=[t,z]$ then there exists $u\leq s,t$.  So we are left with the case $s^{\downarrow}\cap t^{\downarrow}\neq \emptyset$.  Since $S$ is a weak semilattice, we can find elements $u_1,\ldots, u_n\in S$ so that $u\leq s,t$ if and only if $u\leq u_i$ for some $i=1,\ldots, n$.  Let $V=X\setminus \left(\bigcup_{i=1}^n X_{u_i^*u_i}\right)$; it is an open set by hypothesis.   If $x\in X_{u_i^*u_i}$ for some $i$, then since $u_i\leq s,t$, it follows $[s,x]=[t,x]$, a contradiction.  Thus $x\in V$.  Define $W=V\cap X_{s^*s}\cap X_{t^*t}$. We claim $(s,W)$ and $(t,W)$ are disjoint neighborhoods of $[s,x]$ and $[t,x]$, respectively.  Indeed, if $[r,z]\in (s,W)\cap (t,W)$, then $[s,z]=[r,z]=[t,z]$ and hence there exists $u\leq s,t$ with $z\in X_{u^*u}$.  But then $u\leq u_i$ for some $i=1,\ldots,n$ and so $z\in X_{u_i^*u_i}$, contradicting that $z\in W\subseteq V$.

Trivially (2) implies (3).  For (3) implies (1), let $s,t\in S$ and suppose $s^{\downarrow}\cap t^{\downarrow}\neq \emptyset$.  Proposition~\ref{semibooleanalgebra} implies that $(s,D(s^*s))\cap (t,D(t^*t))$ is compact.  But clearly \[(s,D(s^*s))\cap (t,D(t^*t))=\bigcup_{u\in s^{\downarrow}\cap t^{\downarrow}} (u,D(u^*u))\] since $[s,x]=[t,x]$ if and only if there exists $u\leq s,t$ with $x\in D(u^*u)$.  By compactness, we may find $u_1,\ldots,u_n\in s^{\downarrow}\cap t^{\downarrow}$ so that \[(s,D(s^*s))\cap (t,D(t^*t))= (u_1,D(u_1^*u_1))\cup \cdots \cup (u_n,D(u_n^*u_n)).\]  We claim that $u_1,\ldots,u_n$ generate the downset $s^{\downarrow}\cap t^{\downarrow}$.  Indeed, if $u\leq s,t$ then $[u,\chi_{(u^*u)^{\uparrow}}]\in (s,D(s^*s))\cap (t,D(t^*t))$ and so $[u,\chi_{(u^*u)^{\uparrow}}]\in(u_i,D(u_i^*u_i))$ for some $i$.  But then $u\leq u_i$ by Lemma~\ref{embedunderlying}.  This completes the proof.
\end{proof}

Examples of semigroups that are weak semilattices include $E$-unitary and $0$-$E$-unitary inverse semigroups~\cite{Lawson,Exel,Paterson,lenz}.  Notice that if $s\in S$, then the map $x\mapsto x^*x$ provides an order isomorphism between $s^{\downarrow}$ and $(s^*s)^{\downarrow}$; the inverse takes $e$ to $se$.  Recall that a poset is called \emph{Noetherian} if it satisfies ascending chain condition on downsets, or equivalent every downset is finitely generated.  If each principal downset of $E(S)$ is Noetherian, it then follows from the above discussion that $S$ is a weak semilattice.  This occurs of course if $E(S)$ is finite, or if each principal downset of $E(S)$ is finite.  More generally, if each principal downset of $E(S)$ contains no infinite ascending chains and no infinite anti-chains, then $S$ is a weak semilattice.

\begin{Example}[A non-Hausdorff groupoid]
The following example can be found in~\cite[Appendix C]{Paterson}.  Define a commutative inverse semigroup $S=\mathbb N\cup \{\infty,z\}$ by inflating $\infty$ to a cyclic group $\{\infty,z\}$ of order $2$ in the example just after Proposition~\ref{descendingchain}.  Here $\mathbb N\cup \{\infty\}$ is the semilattice considered there with $0<i<\infty$ for $i\geq 1$ and all other elements incomparable.  One has $\{\infty,z\}$ is a cyclic group of order $2$ with non-trivial element $z$.  The element $z$ acts as the identity on $\mathbb N$.  Then $\infty^{\downarrow}\cap z^{\downarrow} = \mathbb N$ is not finitely generated as a downset, in fact the positive naturals are an infinite anti-chain of maximal elements.  It follows $S$ is not a weak semilattice and so $\mathscr G(S)$ is not Hausdorff. Moreover, the compact open slice $(z,D(\infty))$ is not closed and hence $\chi_{(z,D(\infty))}$ is a discontinuous element of $K\mathscr G(S)$.
\end{Example}

\subsection{Covariant representations}
The purpose of this subsection is to indicate that if $S$ is an inverse semigroup with a non-degenerate action on a locally compact boolean space $X$, then $K(S\ltimes X)$ can be thought of as a cross product $KX\rtimes S$.  Let us make this precise.

\begin{Def}[Covariant representation]\label{covariant}
 Let $\theta\colon S\to I_X$ be an ample action of an inverse semigroup $S$.  A \emph{covariant representation} of the dynamical system $(\theta,S,X)$ on a $K$-algebra $A$ is a pair $(\pi,\sigma)$ where $\pi\colon KX\to A$ is a $K$-algebra homomorphism and $\sigma\colon S\to A$ is a homomorphism such that:
\begin{enumerate}
 \item $\pi(f\theta_{s^*}) = \sigma(s)\pi(f)\sigma(s^*)$ for $f\in KX_{s^*s}$;
\item $\pi(\chi_{X_e}) = \sigma(e)$ for $e\in E(S)$.
\end{enumerate}
\end{Def}

See~\cite{Exel,Paterson} for more on covariant representations in the analytic context.
It turns out that covariant representations are in bijection with $K$-algebra homomorphisms in the Hausdorff context and so $K(S\ltimes X)$ has the universal property of a cross product $KX\rtimes S$.  Probably this remains true in the non-Hausdorff case as well.

\begin{Prop}\label{newamplestuff}
Let $\theta\colon S\to I_X$ be an ample action and put $\mathscr G=S\ltimes X$.  Then $\Gamma = \{(s,U)\mid  U\in B(X_{s^*s})\}$ is an inverse subsemigroup of $\mathscr G^a$ which is a basis for the topology of $\mathscr G$ and such that $E(\Gamma)$ generates $B(X)$ as a generalized boolean algebra. Consequently, $K\mathscr G$ is spanned by all characteristic functions $\chi_{(s,U)}$ with $U\in B(X_{s^*s})$.
\end{Prop}
\begin{proof}
Lemma~\ref{slicehom} implies that $\Gamma$ is an inverse subsemigroup of $\mathscr G^a$.  We already know it is a basis for the topology of $\mathscr G$.  If $U\in B(X)$, then since $\bigcup_{e\in E(S)} X_e=X$, compactness of $U$ implies $U\subseteq X_{e_1}\cup\cdots \cup X_{e_n}$ for some $e_1,\ldots, e_n\in E(S)$.  Then $X_{e_i}\cap U\in B(X_{e_i})$, for $i=1,\ldots,n$ and $U=(X_{e_1}\cap U)\cup \cdots \cup (X_{e_n}\cap U)$.  Since $X_{e_i}\cap U\in E(\Gamma)$ for all $i$, we conclude that $E(\Gamma)$ generates $B(X)$ as a generalized boolean algebra.  The final statement now follows from Proposition~\ref{smgroupgen}.
\end{proof}

We are now almost ready to establish the equivalence between covariant representations and $K$-algebra homomorphisms for the case of a Hausdorff groupoid of germs. First we recall some measure-theoretic definitions.

\begin{Def}[Semiring]
A collection $S$ of subsets of a set $X$ is called a \emph{semiring of subsets} if:
\begin{enumerate}
 \item $\emptyset\in S$;
\item $S$ is closed under pairwise intersections;
\item If $A,B\in S$, then $A\setminus B$ is a finite disjoint union of elements of $S$.
\end{enumerate}
\end{Def}

In the context of measure theory, a generalized boolean algebra of subsets is usually called a ring of subsets.
In this language a standard measure theoretic result is that the generalized boolean algebra generated by a semiring $S$ consists precisely of the finite disjoint unions of elements of $S$.

\begin{Def}[Additive function]
Let $S$ be a collection of subsets of $X$ and $A$ an abelian group.  Then a function $\mu\colon S\to A$ is said to be \emph{additive} if whenever $A,B\in S$ are disjoint and $A\cup B\in S$, then $\mu(A\cup B) = \mu(A)+\mu(B)$.
\end{Def}

The following measure-theoretic lemma goes back to von Neumann.

\begin{Lemma}[Extension principle]\label{extension}
Let $S$ be a semiring of subsets of $X$ and suppose $\mu\colon S\to A$ is an additive function to an abelian group $A$.  Then there is a unique additive function $\mu'\colon R(S)\to A$ extending $\mu$ where $R(S)$ is the generalized Boolean algebra (or ring of subsets) generated by $S$.
\end{Lemma}

We are now ready for the main result of this subsection.

\begin{Thm}
Let $\theta\colon S\to I_X$ be an ample action such that $S\ltimes X$ is Hausdorff and let $A$ be a $K$-algebra.  Then there is a bijection between $K$-algebra homomorphisms $\p\colon K(S\ltimes X)\to A$ and covariant representations $(\pi,\sigma)$ of $(\theta,S,X)$ on $A$.
\end{Thm}
\begin{proof}
Set $\mathscr G=S\ltimes X$.   First suppose that $\p\colon K\mathscr G\to A$ is a $K$-algebra homomorphism.  Notice that since $X=\mathscr G^0$ is an open subspace of $\mathscr G$, it follows that $KX$ is a subspace of $K\mathscr G$.  In fact, it is a subalgebra with pointwise product (say by Proposition~\ref{convolutionwelldefined}).  So define $\pi=\p|_{KX}$.  For $s\in S$, put $\wh s=\chi_{(s,X_{s^*s})}$ and define $\sigma(s)=\p(\wh s)$.  Then $\sigma$ is clearly a homomorphism.  Let us verify the two axioms for covariant representations. The second axiom is immediate from the definitions since if $e$ is an idempotent of $S$, then $(e,X_e)$ is the open subset of $\mathscr G^0$ corresponding to $X_e$ under our identification of $\mathscr G^0$ with $X$.

Suppose $f\in KX_{s^*s}$.  Then we compute
\begin{equation}\label{aconjugationformula}
\wh s\ast f\ast \wh {s^*}(x) = \sum_{y\in d\inv d(x)}\wh s(xy\inv)\sum_{z\in d\inv d(y)} f(yz\inv)\wh {s^*}(z).
\end{equation}
Since $f$ has support in $X_{s^*s}$, it follows that to get a non-zero value we must have $y=z$ and $r(y)\in X_{s^*s}$.  Moreover, $d(x)=d(y)=d(z)\in X_{ss^*}$ and $y=z=[s^*,d(x)]$.  Also to obtain a non-zero value we need $xy\inv = x[s,s^*d(x)]=[s,s^*d(x)]$.  Thus $x=d(x)\in X_{ss^*}$ and $yz\inv = r(y) = r(z)=s^*x$.  So \eqref{aconjugationformula} is $0$ if $x\notin X_{ss^*}$ and otherwise is $f(s^*x)$,  This implies $\wh s\ast f\ast \wh{s^*} = f\theta_{s^*}$ and so \[\sigma(s)\pi(f)\sigma(s^*)=\p(\wh s\ast f\ast \wh{s^*}) = \pi(f\theta_{s^*})\] establishing (1) of Definition~\ref{covariant}.

Conversely, suppose $(\pi,\sigma)$ is a covariant representation on $A$. Fix $s\in S$. Observe that $B(X_{s^*s})\cong B((s,X_{s^*s}))$ via the map $U\mapsto (s,U)$. Define a map $\mu_s\colon B((s,X_{s^*s}))\to A$ by $\mu_s((s,U)) = \sigma(s)\pi(\chi_U)$.  We claim that $\mu_s$ is additive.  Indeed, if $U,V\in B(X_{s^*s})$ are disjoint, then $\chi_{U\cup V} = \chi_U+\chi_V$ so
\begin{align*}
\mu_s((s,U)\cup (s,V)) &= \sigma(s)\pi(\chi_{U\cup V})= \sigma(s)\pi(\chi_U)+\sigma(s)\pi(\chi_V)\\ &=\mu_s((s,U))+\mu_s((s,V)).
\end{align*}

Next we claim that if $U\in B((s,X_{s^*s}))\cap B((t,X_{t^*t}))$, then $\mu_s(U)=\mu_t(U)$.  Put $V=d(U)$. Then $U=\{[s,x]\mid x\in V\}=\{[t,x]\mid x\in V\}$.  For each $x\in V$, we can find an element $u_x\in X$ so that $u_x\leq s,t$ and $x\in X_{u_x^*u_x}$.  By compactness of $V$, we conclude that there exist $v_1,\ldots, v_n\leq s,t$ so that if $V_i = X_{v_i^*v_i}\cap V$, then $V=V_1\cup\cdots\cup V_n$.  Since $B(V)$ is a boolean algebra, we can refine this to a disjoint union $V=U_1\cup\cdots\cup U_m$ such that there are elements $u_1,\ldots, u_m\leq s,t$ (not necessarily all distinct) so that $U_i\subseteq X_{u_i^*u_i}\cap V$ and $U_i\in B(X)$, for $i=1,\ldots, m$ (cf.\ Step 2 of Theorem~\ref{presentation}).  Then $U = (u_1,U_1)\cup\cdots\cup (u_m,U_m)$ as a disjoint union.  By additivity of $\mu_s$ and $\mu_t$, it therefore suffices to show that if $w\leq s,t$ and $W\subseteq X_{w^*w}$, then $\mu_s((w,W))=\mu_t((w,W))$.  Equivalently, we must show $\sigma(s)\pi(\chi_W)=\sigma(t)\pi(\chi_W)$.  Now we compute \[\sigma(s)\pi(\chi_W) = \sigma(s)\pi(\chi_{X_{w^*w}}\chi_W) = \sigma(sw^*w)\pi(\chi_W) = \sigma(w)\pi(\chi_W)\] and similarly $\sigma(t)\pi(\chi_W)=\sigma(w)\pi(\chi_W)$.  This concludes the proof that $\mu_s(U)=\mu_t(U)$.

Let $\Gamma$ be as in Proposition~\ref{newamplestuff}; notice that $\Gamma=\bigcup_{s\in S}B((s,X_{s^*s}))$.  Then there is a well-defined function $\mu\colon \Gamma\to A$ given by $\mu((s,U))  =\mu_s((s,U))$.  Moreover, $\mu$ is additive since the disjoint union of two elements of $\Gamma$ belongs to $\Gamma$ if and only if they both belong to $B(s,X_{s^*s})$ for some $s$ and then one applies the additivity of $\mu_s$.  Since $\mathscr G$ is Hausdorff, Proposition~\ref{semibooleanalgebra} shows that $\mathscr G^a$ is closed under intersection and relative complement.  Now clearly, $\Gamma$ is a semiring of subsets of $\mathscr G$ since it is a downset in $\mathscr G^a$ and hence closed under finite intersections and relative complements.  On the other hand, since $\Gamma$ is a basis for the topology of $\mathscr G$ and each element of $\mathscr G^a$ is compact, it follows that $\mathscr G^a$ is contained in the generalized boolean algebra generated by $\Gamma$.  Lemma~\ref{extension} now provides a well-defined additive function $\mu'\colon \mathscr G^a\to A$ extending $\mu$.  To show that $\mu'$ is a semigroup homomorphism, it suffices to show that its restriction $\mu$ to the subsemigroup $\Gamma$ is a homomorphism since $\mu'$ is additive and the product distributes over those disjoint unions that exist in $\mathscr G^a$.

So suppose $(s,U),(t,V)\in \Gamma$.  Then we compute
\begin{align*}
\mu((s,U))\mu((t,V)) &= \sigma(s)\pi(\chi_U)\sigma(t)\pi(\chi_V)\\
&= \sigma(s)\pi(\chi_U)\sigma(tt^*)\sigma(t)\pi(\chi_V) \\
&= \sigma(s)\sigma(t)\sigma(t^*)\pi(\chi_{U\cap X_{tt^*}})\sigma(t)\pi(\chi_V) \\
&= \sigma(st)\pi(\chi_{U\cap X_{tt^*}}\theta_t\cdot \chi_V).
\end{align*}
But $[\chi_{U\cap X_{tt^*}}\theta_t(x)]\chi_V(x)$ is $1$ if and only if $x\in V$ and $tx\in U$, which occurs if and only if $x\in t^*(U\cap tV)$.  Indeed, if $x\in V$ and $tx\in U$, then $tx\in tV$ and so $x=t^*tx\in t^*(U\cap tV)$.  Conversely, $x\in t^*(U\cap tV)$ implies there exists $v\in V$ so that $tv\in U$ and $x=t^*tv=v$.  Thus $x\in V$ and $tx=tv\in U$.  We conclude $\mu((s,U))\mu((t,V)) = \sigma(st)\pi(\chi_{t^*(U\cap tV)})$.
On the other hand, $(s,U)(t,V) = (st,t^*(U\cap tV))$ by Proposition~\ref{slicehom} and so $\mu((s,U)(t,V))= \sigma(st)\pi(\chi_{t^*(U\cap tV)})$.  This completes the proof that $\mu$, and hence $\mu'$, is a homomorphism.

Since $\mu'\colon \mathscr G^a\to A$ is an additive semigroup homomorphism, Theorem~\ref{presentation} provides a homomorphism $\p\colon K\mathscr G\to A$ satisfying $\p(\chi_{(s,U)}) = \sigma(s)\pi(\chi_{U})$.  In particular, $\p(\chi_{(s,X_{s^*s})}) = \sigma(s)\pi(\chi_{X_{s^*s}}) = \sigma(s)\sigma(s^*s)=\sigma(s)$ for $s\in S$.  On the other hand, if $U\in B(X)$, we can write $U=U_1\cup\cdots \cup U_n$ a disjoint union where $U_i\subseteq X_{e_i}$ for some idempotent $e_i$ using compactness of $U$ and the $X_e$, as well as the non-degeneracy of the action.  Then
\begin{align*}
\p(\chi_U) &= \p(\chi_{U_1})+\cdots+\p(\chi_{U_n})= \sigma(e_1)\pi(\chi_{U_1})+\cdots+\sigma(e_n)\pi(\chi_{U_n}) \\&=  \pi(\chi_{U_1})+\cdots+\pi(\chi_{U_n}) = \pi(\chi_U)
\end{align*}
since $\sigma(e_i)\pi(\chi_{U_i}) = \pi(\chi_{X_{e_i}\cap U_i}) = \pi(\chi_{U_i})$.  This proves that the two constructions in this proof are inverse to each other, establishing the desired bijection.
\end{proof}

\section{The isomorphism of algebras}
The main theorem of this section says that if $K$ is any unital commutative ring endowed with the discrete topology and $S$ is an inverse semigroup, then $KS\cong K\mathscr G(S)$.  The idea is to combine Paterson's proof for $C^*$-algebras~\cite{Paterson} with the author's proof for inverse semigroups with finitely many idempotents~\cite{mobius1,mobius2}. Recall that the semigroup algebra $KS$ is the free $K$-module with basis $S$ equipped with the usual convolution product
\[\sum_{s\in S}c_ss\cdot \sum_{t\in S} d_tt = \sum_{s,t\in S}c_sd_tst.\]  In the case that $K=\mathbb C$, we make $\mathbb CS$ into a $\ast$-algebra by taking \[\left(\sum_{s\in S} c_ss\right)^* = \sum_{s\in S} \ov{c_s}s^*.\]

We begin with a lemma that is an easy consequence of Rota's theory of M\"obius inversion~\cite{Stanley,Burnsidealgebra}.
\begin{Lemma}\label{easymobius}
Let $P$ be a finite poset.  Then the set $\{\chi_{p^{\downarrow}}\mid p\in P\}$ is a basis for $K^P$.
\end{Lemma}
\begin{proof}
 The functions $\{\delta_p\mid p\in P\}$ form the standard basis for $K^P$.  With respect to this basis \[\chi_{p^{\downarrow}} = \sum_{q\leq p} \delta_q.\]  Thus by M\"obius inversion, \[\delta_p = \sum_{q\leq p}\chi_{q^{\downarrow}}\mu(q,p)\] where $\mu$ is the M\"obius function of $P$.  This proves the lemma.
\end{proof}

Alternatively, one can order $P=\{p_1,\ldots, p_n\}$ so that $p_i\leq p_j$ implies $i\leq j$.  The linear transformation $p_i\mapsto \sum_{p_j\leq p_i} p_j$ is given by a unitriangular integer matrix and hence is invertible over any commutative ring with unit.

As a corollary, we obtain the following infinitary version.

\begin{Cor}\label{easymobius2}
Let $P$ be a poset.  Then the set $\{\chi_{p^{\downarrow}}\mid p\in P\}$ in $K^P$ is linearly independent.
\end{Cor}
\begin{proof}
It suffices to show that, for any finite subset $F\subseteq P$, the set $F'=\{\chi_{p^{\downarrow}}\mid p\in F\}$ is linearly independent.  Consider the projection $\pi\colon K^P\to K^F$ given by restriction.  Lemma~\ref{easymobius} implies that $\pi(F')$ is a basis for $K^F$.  We conclude that $F'$ is linearly independent.
\end{proof}

We are now ready for one of our main theorems, which generalizes the results of~\cite{mobius1,mobius2} for the case of an inverse semigroup with finitely many idempotents.

\begin{Thm}\label{mainthm}
Let $K$ be a commutative ring with unit and $S$ an inverse semigroup.  Then the homomorphism $\Theta\colon S\to K\mathscr G(S)$ given by $\Theta(s) = \chi_{(s,D(s^*s))}$ extends to an isomorphism of $KS$ with $K\mathscr G(S)$.  Moreover, when $K=\mathbb C$ the map $\Theta$ extends to a $\ast$-isomorphism.
\end{Thm}
\begin{proof}
Proposition~\ref{amplestuff} establishes everything except the injectivity of the induced homomorphism $\Theta\colon KS\to K\mathscr G(S)$.  This amounts to showing that the set of elements $\{\Theta(s)\mid s\in S\}$ is linearly independent.  The key idea is to exploit the dense embedding of  the underlying groupoid of $S$ as a subgroupoid of $\mathscr G(S)$ from Lemma~\ref{embedunderlying}.  More precisely, Lemma~\ref{embedunderlying}  provides an injective mapping $S\to \mathscr G(S)$ given by $s\mapsto [s,\chi_{(s^*s)^{\uparrow}}]=\wh{s}$.

Define a $K$-linear map $\psi\colon K\mathscr G\to K^S$ by $\psi(f)(s) = f(\wh s)$.  Then, if $t\in S$, one has that $\psi(\Theta(t)) = \chi_{t^{\downarrow}}$ by Lemma~\ref{embedunderlying}.   Corollary~\ref{easymobius2} now implies that $\psi(\Theta(S))$ is linearly independent and hence $\Theta(S)$ is linearly independent, completing the proof.
\end{proof}

In the case that $E(S)$ is finite, one has that $\mathscr G(S)$ is the underlying groupoid and so we recover the following result of the author~\cite{mobius1,mobius2} (the final statement of which a proof can be found in these references).

\begin{Cor}
Let $S$ be an inverse semigroup so that $E(S)$ is finite and suppose $K$ is a commutative ring with unit.  Let $\ov S= \{\ov s\mid s\in S\}$ be a disjoint copy of $S$.  Endow $K\ov S$ with a multiplicative structure by putting \[\ov s\cdot \ov t= \begin{cases} \ov{st} & s^*s=tt^*\\ 0 & \text{else.}\end{cases}\]  Then there is an isomorphism from $KS$ to $K\ov S$ sending $s$ to $\sum_{t\leq s}\ov t$.  Hence $KS$ is isomorphic to a finite direct product of finite dimensional matrix algebras over the group algebras of maximal subgroups of $S$.
\end{Cor}

The special case where $S=E(S)$ was first proved by Solomon~\cite{Burnsidealgebra}.
As a consequence of Theorem~\ref{mainthm} and Proposition~\ref{unital}, we obtain the following topological criterion for an inverse semigroup algebra to have a unit as well as a characterization of the center of $KS$.

\begin{Cor}
Let $K$ be a commutative ring with unit and $S$ an inverse semigroup.  Then $KS$ has a unit if and only if $\wh{E(S)}$ is compact.  The center of $KS$ is the space of class functions on $\mathscr G(S)$.
\end{Cor}

%%%
Let $FIM(X)$ be the free inverse monoid on a set $X$ with $|X|\geq 1$.  Crabb and Munn described the center of $KFIM(X)$ in~\cite{Munncentre}.  We give a topological proof of their result using that $KFIM(X)\cong K\mathscr G(FIM(X))$ by describing explicitly the class functions on $\mathscr G(FIM(X))$.   The reader should consult~\cite{Lawson} for the description of elements of $FIM(X)$ as Munn trees.

\begin{Thm}
Let $X$ be a non-empty set.  Then if $|X|=\infty$, the center of $KFIM(X)$ consists of the scalar multiples of the identity.  Otherwise, $Z(KFIM(X))$ is a subalgebra of $KE(FIM(X))$ isomorphic to the algebra of functions $f\colon FIM(X)/{\mathscr D}\to K$ spanned by the finitely supported functions and the constant map to $1$.
\end{Thm}
\begin{proof}
The structure of $\mathscr G(FIM(X))$ is well known cf.~\cite[Chapter 4]{Paterson} or~\cite{strongmorita}.  Let $F(X)$ be the free group on $X$ and denote its Cayley graph by $\Gamma$.  Let $\mathscr T$ be the space of all subtrees of $\Gamma$ containing $1$.  Viewing a subtree as the characteristic function of a map $V(\Gamma)\cup E(\Gamma)\to \{0,1\}$, we give $\mathscr T$ the topology of pointwise convergence.  It is easy to see that $\mathscr T$ is a closed subspace of $\{0,1\}^{V(\Gamma)\cup E(\Gamma)}$.  The space $\mathscr T$ is homeomorphic to the character space of $E(FIM(X))$.  The groupoid $\mathscr G=\mathscr G(FIM(X))$ consists of all pairs $(w,T)\in F(X)\times \mathscr T$ so that $w\in V(T)$.  The topology is the product topology.  In particular, the pairs $(w,T)$ with $T$ finite are dense in $\mathscr G$.  One has $d(w,T) = w\inv T$, $r(w,T)= T$ and the product is defined by $(w,T)(w',T') = (ww',T)$ if $w\inv T=T'$.  The inverse is given by $(w,T)\inv = (w\inv ,w\inv T)$.  The groupoid $\mathscr G$ is Hausdorff and so $K\mathscr G$ consists of continuous functions with compact support in the usual sense.

Let $f$ be a class function.
We claim that the support of $f$ is contained in $\mathscr G^0=\mathscr T$.   Since $f$ is continuous with compact support and $K$ is discrete, it follows that $f\inv(K\setminus\{0\})$ is compact open and hence the support of $f$.  Thus   $f\inv(K\setminus\{0\})$ is of the form $(\{w_1\}\times C_1)\cup\cdots \cup (\{w_m\}\times C_m)$ where the $C_i$ are compact open subsets of $\mathscr T$.   Suppose that $(w,T)\in \{w_i\}\times C_i$ with $w\neq 1$, and so in particular $w=w_i$.   As $\{w_i\}\times C_i$ is open and the finite trees are dense, there is a finite tree $T'$ containing $1$ and $w$ so that $(w,T')$ belongs  $\{w_i\}\times C_i$.  But no finite subtree of $\Gamma$ is invariant under a non-trivial element of $F(X)$, so $d(w,T')=w\inv T'\neq T'=r(w,T')$ and hence $f(w,T')=0$ as $f$ is a class function.  This contradiction shows that the support of $f$ is contained in $\mathscr G^0$.

Thus we may from now on view $f$ as a continuous function with compact support on $\mathscr T$.
Next observe that if $f(T)=c$ for a tree $T$ and $u\in V(T)$,  then $f(u\inv T)=c$. Indeed, $d(u,T) = u\inv T$ and $r(u,T) = T$. Thus $f(u\inv T) = f((u,T)\inv (1,T)(u,T)) = f(T)=c$ as $f$ is a class function.

Let $f$ be a class function.  Since $K$ is discrete, $f=c_1\chi_{U_1}+\cdots+c_k\chi_{U_k}$ where $U_1,\ldots, U_k$ are non-empty disjoint compact open subsets of $\mathscr T$ and $c_1,\ldots, c_k$ are distinct non-zero elements of $K$.  It is easy to see then that $\chi_{U_1},\ldots, \chi_{U_k}$ must then be class functions.  In other words, the class functions are spanned by the characteristic functions $\chi_U$ of compact open subsets $U$ of $\mathscr T$ so that $T\in U$ implies $u\inv T\in U$ for all vertices $u$ of $T$.

Suppose first that $X$ is infinite.  We claim that no proper non-empty compact open subset $U$ of $\mathscr T$ has the above property.  Suppose this is not the case.  Then there is a subtree $T_0$ that does not belong to $U$.  Since $X$ is infinite and $U$ is determined by a boolean formula which is a finite disjunction of allowing/disallowing finitely many vertices and edges of $\Gamma$, there is a letter $x\in X$ so that no vertex or edge in the boolean formula determining $U$ involves the letter $x$. Let $T\in U$.  Then $T\cup xT_0\in U$ since the edges and vertices appearing in $xT_0$ are irrelevant in the definition of $U$ and $T\in U$.  Thus $x\inv(T\cup xT_0) =x\inv T\cup T_0\in U$.  But since the edges and vertices appearing $x\inv T$ again are irrelevant to the boolean formula defining $U$, we must have $T_0\in U$, a contradiction.

Next suppose that $X$ is finite.  First note that the finite trees form a discrete subspace of $\mathscr T$.  Indeed,  if $T$ is a finite subtree of $\Gamma$ containing $1$, then since $X$ is finite there are only finitely many edges of $\Gamma$ incident on $T$ that do not belong to it.  Then the neighborhood of $T$ consisting of all subtrees containing $T$ but none of these finitely many edges incident on $T$ consists only of $\{T\}$.  So if $T$ is a finite tree, then $U_T=\{v\inv T\mid v\in T\}$ is a finite open subset of $\mathscr T$ and hence its characteristic function belongs to the space of class functions.  We claim that the space of class functions has basis the functions of the form $\chi_{U_T}$  with $T$ a finite subtree of $\Gamma$ containing $1$ and of the identity $\chi_{\mathscr G^0}$.  This will prove the theorem since the sets of the form $U_T$ are in bijection with the $\D$-classes of $S$.

So let $U$ be a compact open set so that $T\in U$ and $u\in T$ implies $u\inv T\in U$.  Suppose that $U$ contains only finite trees.  Since the finite trees are discrete in $\mathscr T$ by the above argument, it follows that $U$ is finite.  The desired claim now follows from the above case. So we may assume that $U$ contains an infinite tree $T$. Since $X$ is finite, it is easy to see that there exists $N>0$ so that $U$ consists of those subtrees of $\Gamma$ whose closed ball of radius $N$ around $1$ belongs to a certain subset $F$ of the finite set of possible closed balls of radius $N$ of an element of $\mathscr T$.  We claim that $U$ contains all infinite trees.  Then applying the previous case to the complement of $U$ proves the theorem.

Suppose first $|X|=1$. Then some translate of the infinite subtree $T$ has closed ball of radius $N$ around $1$ a path of length $2N$ centered around $1$ and so this closed ball belongs to $F$.  However, all infinite subtrees of $\Gamma$ have this closed ball as the ball of radius $N$ around $1$ for some translate.  Thus $U$ contains all the infinite subtrees.

Next suppose $|X|\geq 2$.  Let $T'$ be an infinite tree with closed ball $B$ of radius $N$ around $1$ and let $v$ be a leaf of $B$ at distance $N$ from $1$ (such exists since $T'$ is infinite and $X$ is finite).  Then there is a unique edge of $B$ with terminal vertex $v$, let us assume it is labeled by $x^{\epsilon}$ with $x\in X$ and $\epsilon =\pm 1$.  Since $T$ is infinite, we can find a vertex $u$ of $T$ at distance $N$ from $1$.  Let $T_0$ be the closed ball of radius $N$ in $T$ around $1$.  Then in $T_0$, the vertex $u$ is the endpoint of a unique edge of $\Gamma$.  If this edge is not labeled by $x^{\epsilon}$, then put $T_1 = T_0\cup uv\inv B$.  Otherwise, choose $y\in X$ with $y\neq x$ and put $T_1 = T_0\cup \{u\xrightarrow{y} uy\}\cup uyv\inv B$.  In either case, the closed balls of radius $N$ around $1$ in $T_1$ and $T$ coincide, and so $T_1\in U$.  But there is a translate $T_2$ of $T_1$ so that the closed ball of radius $N$ about $1$ in $T_2$ is exactly $B$.  Thus $B\in F$ and so $T'\in U$.  This completes the proof of the theorem.
\end{proof}

Let $\mathscr G$ be an ample groupoid and $C_c(\mathscr G)$ be the usual algebra of continuous (complex-valued) functions with compact support on $\mathscr G$~\cite{Exel,Paterson}.  Notice that $\mathbb C\mathscr G$ is a subalgebra of $C_c(\mathscr G)$ since any continuous function to $\mathbb C$ with respect to the discrete topology is continuous in the usual topology.  Let $\|\cdot \|$ be the $C^*$-norm on $C_c(\mathscr G)$~\cite{Exel,Paterson}.  The following is essentially~\cite[Proposition 2.2.7]{Paterson}.

\begin{Prop}
Let $\mathscr G$ be an ample groupoid with $\mathscr G^0$ countably based.  Then $C^*(\mathscr G)$ is the completion of $\CC\mathscr G$ with respect to its own universal $C^*$-norm.
\end{Prop}
\begin{proof}
To prove the theorem, it suffices to verify that any non-degenerate $\ast$-representation $\pi\colon \CC\mathscr G\to \mathscr B(H)$ with $H$ a (separable) Hilbert space extends uniquely to $C_c(\mathscr G)$.  Indeed, this will show that $\CC\mathscr G$ is dense in the $C^*$-norm on $C_c(\mathscr G)$ and that the restriction of the $C^*$-norm on $C_c(\mathscr G)$ to $\CC\mathscr G$ is its own $C^*$-norm. Suppose $V$ is an open neighborhood in $\mathscr G$ and $f\in C_c(V)$. Then since $\mathscr G$ has a basis of compact open subsets, we can cover the compact support of $f$ by a compact open subset $U$.    Thus it suffices to define the extension of $\pi$ for any $f\in C(U)$ where $U$ is a compact open subset of $\mathscr G$.   Since $U$ has a basis of compact open subsets, the continuous functions on $U$ with respect to the discrete topology separate points. The Stone-Weierstrass Theorem implies that we can find a sequence $f_n\in \CC\mathscr G$ so that $\|f_n-f\|_{\infty}\to 0$.
Now the argument of~\cite[Proposition 3.14]{Exel} shows that if $g\in C(U)\cap \CC\mathscr G$, then $\|\pi(g)\|\leq \|g\|_{\infty}$.  It follows that $\pi(f_n)$ is a Cauchy sequence and so has a limit that we define to be $\pi(f)$.  It is easy to check that $\pi(f)$ does not depend on the Cauchy sequence by a simple interweaving argument.  The reader should verify that $\pi$ is a representation of $C_c(\mathscr G)$ cf.~\cite{Paterson}.  It is the unique extension since if $\pi'$ is an extension, then~\cite[Proposition 3.14]{Exel} implies $\|\pi'(g)\|\leq \|g\|_{\infty}$ for $g\in C(U)$.  Thus $\|\pi(f_n)-\pi'(f)\|\to 0$ and so $\pi'(f)=\pi(f)$.
\end{proof}

We now recover an important result of Paterson~\cite{Paterson}.
\begin{Cor}[Paterson]
Let $S$ be an inverse semigroup.  Then there is an isomorphism $C^*(S)\cong C^*(\mathscr G(S))$ of universal of $C^*$-algebras.
\end{Cor}

Paterson also established an isomorphism of reduced $C^*$-algebras~\cite{Paterson}.

\section{Irreducible representations}
  Our aim is to construct the finite dimensional irreducible representations of an arbitrary inverse semigroup over a field $K$ and determine when there are enough such representations to separate points.  Our method can be viewed as a generalization of the groupoid approach of the author~\cite{mobius1,mobius2} to the classical theory of Munn and Ponizovsky~\cite{CP} for inverse semigroups with finitely many idempotents.  See also~\cite{myirreps}.  We begin by describing all the finite dimensional irreducible representations of an ample groupoid.   The desired result for inverse semigroups is deduced as a special case via the universal groupoid.

In fact, much of what we do works over an arbitrary commutative ring with unit $K$, which shall remain fixed for the section. Let $A$ be a $K$-algebra.  We say that an $A$-module $M$ is \emph{non-degenerate} if $AM=M$.   We consider here only the category of non-degenerate $A$-modules.  So when we write the words ``simple module,'' this should be understood as meaning non-degenerate simple module.  Note that if $A$ is unital, then an $A$-module is non-degenerate if and only if the identity of $A$ acts as the identity endomorphism.   A representation of $A$ is said to be \emph{non-degenerate} if the corresponding module is non-degenerate.  As usual, there is a bijection between isomorphism classes of (finite dimensional) simple $A$-modules and equivalence classes of non-degenerate (finite dimensional) irreducible representations of $A$ by $K$-module endomorphisms.

\subsection{Irreducible representations of ample groupoids}
Fix an ample groupoid $\mathscr G$.  Then one can speak about the orbit of an element of $\mathscr G^0$ and its isotropy group.

\begin{Def}[Orbit]
Define an equivalence relation on $\mathscr G^0$ by setting $x\sim y$ if there is an arrow $g\in \mathscr G$ such that $d(g)=x$ and $r(g)=y$.  An equivalence class will be called an \emph{orbit}.   If $x\in \mathscr G^0$, then \[G_x=\{g\in \mathscr G\mid d(g)=x=r(g)\}\] is called the \emph{isotropy group} of $\mathscr G$ at $x$.  It is well known and easy to verify that up to conjugation in $\mathscr G$ (and hence isomorphism) the isotropy group of $x$ depends only on the orbit of $x$.  Thus we may speak abusively of the isotropy group of the orbit.
\end{Def}

To motivate the terminology, if $G$ is a group acting on a space $X$, then the orbit of $x\in X$ in the groupoid $G\ltimes X$ is exactly the orbit of $x$ in the usual sense.  Moreover, the isotropy group of $x$ in $G\ltimes X$ is isomorphic to the stabilizer in $G$ of $x$ (i.e., the usual isotropy group).

\begin{Rmk}[Underlying groupoid]
If $S$ is an inverse semigroup and $\mathscr G$ is its underlying groupoid, then the orbit of $e\in E(S)$ is precisely the set of idempotents of $S$ that are $\mathscr D$-equivalent to $e$ and the isotropy group $G_e$ is the maximal subgroup at $e$~\cite{Lawson}.
\end{Rmk}

In an ample groupoid, the orbit of a unit is its orbit under the action of $\mathscr G^a$ described earlier.  Indeed, given $d(g)=x$ and $r(g)=y$, choose a slice $U\in \mathscr G^a$ containing $g$.  Clearly, we have $Ux=y$.  Conversely, if $U$ is a slice with $Ux=y$, then we can find $g\in U$ with $d(g)=x$ and $r(g)=y$.

The following lemma seems worth noting, given the importance of finite orbits in what follows.  One could give a topological proof along the lines of~\cite{inversetop} since this is essentially the same argument used in computing the fundamental group of a cell complex.

\begin{Lemma}\label{fundamentalgroup}
Let $S$ be an inverse semigroup with generating set $A$ acting non-degenerately on a space $X$ and let $\mathscr O$ be the orbit of $x\in X$.  Fix, for each $y\in \mathscr O$, an element $p_y\in S$ so that $p_yx=y$ where we choose $p_x$ to be an idempotent.  For each pair $(a,y)\in A\times \mathscr O$ such that $ay\in \mathscr O$, define $g_{a,y} = [p_{ay}^*ap_y,x]$.  Then the isotropy group $G_x$ is generated by the set of elements $\{g_{a,y}\mid a\in A, y,ay\in \mathscr O\}$.  In particular, if $A$ and $\mathscr O$ are finite, then $G_x$ is finitely generated.
\end{Lemma}
\begin{proof}
First note that if $ay\in \mathscr O$, then $p_{ay}^*ap_yx = p_{ay}^*ay =x$ and so $g_{a,y}\in G_x$.  Let us define, for $a\in A$ and $y\in \mathscr O$ with $a^*y\in \mathscr O$, the element $g_{a^*,y} = [p_{a^*y}^*a^*p_y,x]\in \mathscr G^a$.  Notice that $g_{a^*,y} = g_{a,a^*y}\inv$.

Suppose that $[s,x]\in G_x$ and write $s=a_n\cdots a_1$ with the $a_i\in A\cup A^*$.  Define $x_i = a_i\cdots a_1x$, for $0=1,\ldots, n$ (so $x_0=x=x_n$) and consider the element $t=(p_{x_n}^*a_np_{x_{n-1}})\cdots (p_{x_2}^*a_2p_{x_1})(p_{x_1}^*a_1p_{x_0})$ of $S$.  Then $t\leq s$ and $tx=sx=x$.  Hence $[s,x]=[t,x] = g_{a_n,x_{n-1}}\cdots g_{a_1,x_0}$, as required.
\end{proof}

Applying this to the Munn representation and the underlying groupoid, we obtain the following folklore result (a simple topological proof can be found in~\cite{inversetop}).

\begin{Cor}\label{fgmaxsubgroup}
Let $S$ be a finitely generated inverse semigroup and $e$ an idempotent whose $\D$-class contains only finitely many idempotents.  Then the maximal subgroup $G_e$ of $S$ at $e$ is finitely generated.
\end{Cor}

We remark that if we consider the spectral action of $S$ on $\wh{E(S)}$ and $e\in E(S)$, then the orbit of $\chi_{e^{\uparrow}}$ is $\{\chi_{f^\uparrow}\mid f\D e\}$ and $G_{\chi_{e^\uparrow}}=G_e$.

Fix $x\in \mathscr G^0$.  Define $L_x=d\inv(x)$ (inverse semigroup theorists should think of this as the $\mathscr L$-class of $x$).  The isotropy group $G_x$ acts on the right of $L_x$ and $L_x/G_x$ is in bijection with the orbit of $x$ via the map $tG_x\mapsto r(t)$.  Indeed, if $s,t\in L_x$, then $r(s)=r(t)$ implies $t\inv s\in G_x$ and of course $s=t(t\inv s)$.  Conversely, every element of $tG_x$ evidently has range $r(t)$.

There is also a natural action of $\mathscr G^a$ on the left of $L_x$ that we shall call, in analogy with the case of inverse semigroups~\cite{CP}, the \emph{Sch\"utzenberger representation} of $\mathscr G^a$ on $L_x$.  If $U\in \mathscr G^a$, then we define a map \[U\cdot\colon L_x\cap r^{-1}(U\inv U)\to L_x\cap r^{-1}(UU\inv)\] by putting $Ut=st$ where $s$ is the unique element of $U$ with $d(s)=r(t)$ (or equivalent, $Ut=y$ where $yt\inv\in U$).  We leave the reader to verify that this is indeed an action of $\mathscr G^a$ on $L_x$ by partial bijections.

There is an alternative construction of $L_x$ and $G_x$ which will be quite useful in what follows.  Let $\til L_x = \{U\in \mathscr G^a\mid x\in U\inv U\}$ and put $\mathscr G^a_x = \{U\in \mathscr G^a\mid Ux=x\}$.  Notice that $\til L_x = \{U\in \mathscr G^a\mid U\cap L_x\neq \emptyset\}$ and  $\mathscr G^a_x=\{U\in \mathscr G^a\mid U\cap G_x\neq \emptyset\}$.  It is immediate that $\mathscr G^a_x$ is an inverse subsemigroup of $\mathscr G^a$ acting on the right of $\til L_x$.  An element of $\til L_x$ intersects $L_x$ in exactly on element by the definition of a slice.

\begin{Lemma}\label{Lxasgerms}
Define a map $\nu\colon \til L_x\to L_x$ by $U\cap L_x = \{\nu(U)\}$.  Then:
\begin{enumerate}
\item $\nu$ is surjective;
 \item $\nu(U)=\nu(V)$ if and only if $U$ and $V$ have a common lower bound in $\til L_x$;
\item $\nu\colon \mathscr G^a_x\to G_x$ is the maximal group image homomorphism.
\end{enumerate}
\end{Lemma}
\begin{proof}
If $t\in L_x$ and $U\in \mathscr G^a$ with $t\in U$, then $U\in \til L_x$ with $\nu(U)=t$.  This proves (1).  For (2),
trivially, if $W\subseteq U,V$ is a common lower bound in $\til L_x$, then $\nu(U)=\nu(W)=\nu(V)$.  Conversely, suppose that $\nu(U)=\nu(V)=t$.  Then $U\cap V$ is an open neighborhood of $t$.  Since $\mathscr G^a$ is a basis for the topology on $\mathscr G$, we can find $W\in \mathscr G^a$ with $t\in W\subseteq U\cap V$.  As $W\in \til L_x$, this yields (2).

Evidentally, $\nu$ restricted to $\mathscr G^a_x$ is a group homomorphism.  By (2), it is the maximal group image since any common lower bound in $\til L_x$ of elements of $\mathscr G^a_x$ belongs to $\mathscr G^a_x$.  This proves (3).
\end{proof}

\begin{Rmk}
In fact, $\nu$ gives a morphism from the right action of $\mathscr G^a_x$ on $\til L_x$ to the right action of $G_x$ on $L_x$.
\end{Rmk}

Consider a free $K$-module $KL_x$ with basis $L_x$.   The right action of $G_x$ on $L_x$ induces a right $KG_x$-module structure on $KL_x$.  Let $T$ be a transversal for $L_x/G_x$.  We assume $x\in T$.

\begin{Prop}\label{free}
The isotropy group $G_x$ acts freely on the right of $L_x$ and hence $KL_x$ is a free right $KG_x$-module with basis $T$.
\end{Prop}
\begin{proof}
It suffices to show that $G_x$ acts freely on $L_x$.  But this is clear since if $t\in L_x$ and $g\in G_x$, then $tg=t$ implies $g=t\inv t=x$.
\end{proof}

We now endow $KL_x$ with the structure of a left $K\mathscr G$-module by linearly extending  the Sch\"utzenberger representation.  Formally, suppose $f\in K\mathscr G$ and $t\in L_x$.  Define
\begin{equation}\label{moduleaction}
ft = \sum_{y\in L_x}f(yt\inv)y.
\end{equation}

\begin{Prop}\label{bimodule}
If $U\in \mathscr G^a$ and $t\in L_x$, then
\begin{equation}\label{schutzrep}
\chi_Ut = \begin{cases} Ut & r(t)\in U\inv U\\
0 & \text{else.}\end{cases}
\end{equation}
Consequently, $KL_x$ is a well-defined $K\mathscr G$-$KG_x$ bimodule.
\end{Prop}
\begin{proof}
Since $K\mathscr G$ is spanned by characteristic functions of elements of $\mathscr G^a$ (Proposition~\ref{characteristicbasis}) in order to show that \eqref{moduleaction} is a finite sum, it suffices to verify \eqref{schutzrep}.  If $r(t)\notin U\inv U$, then $\chi_U(yt\inv)=0$ for all $y\in L_x$.  On the other hand, suppose $r(t)\in U\inv U$ and say $r(t)=d(s)$ with $s\in U$.  Then $Ut=st$  and $y=st$ is the unique element of $L_x$ with $yt\inv\in U$.  Hence $\chi_Ut=y=st=Ut$.   Since the Sch\"utzenberger representation is an action, it follows that \eqref{moduleaction} gives a well-defined left module structure to $KL_x$.

To see that $KL_x$ is a bimodule, it suffices to verify that if $f\in \mathscr G^a$, $g\in G_x$ and $t\in L_x$, then $(ft)g=f(tg)$.  This is shown by the following computation:
\begin{align*}
(ft)g &= \left(\sum_{y\in L_x}f(yt\inv)y\right)g = \sum_{y\in L_x}f(yt\inv)yg\\
f(tg) & = \sum_{z\in L_x}f(zg\inv t\inv)z = \sum_{y\in L_x}f(yt\inv)yg
\end{align*}
where the final equality of the second equation is a consequence of the change of variables $y=zg\inv$.
\end{proof}

We are now prepared to define induced modules.

\begin{Def}[Induction]
For $x\in \mathscr G^0$ and $V$ a $KG_x$-module, we define the corresponding \emph{induced} $K\mathscr G$-module to be $\Ind_x(V)=KL_x\otimes_{KG_x}V$.
\end{Def}

The induced modules coming from elements of the same orbit coincide.  More precisely, if $y$ is in the orbit of $x$ with, say, $d(s)=x$ and $r(s)=y$ and if $V$ is a $KG_x$-module, then $V$ can be made into a $KG_y$-module by putting $gv=s\inv gsv$ for $g\in G_y$ and $v\in V$. Then $\Ind_x(V)\cong \Ind_y(V)$ via the map $t\otimes v\mapsto ts\inv \otimes v$ for $t\in L_x$ and $v\in V$.

The following result, and its corollary, will be essential to studying induced modules.

\begin{Prop}\label{createprojections}
Let $t,u,s_1,\ldots, s_n\in L_x$ with $s_1,\ldots, s_n\notin tG_x$.  Then there exists $U\in \mathscr G^a$ so that $\chi_Ut = u$ and $\chi_Us_i=0$ for $i=1,\ldots,n$.
\end{Prop}
\begin{proof}
The set $B(\mathscr G^0)$ is a basis for the topology of $\mathscr G^0$.  Hence we can find $U_0\in B(\mathscr G^0)$ so that $U_0\cap \{r(t),r(s_1),\ldots,r(s_n)\}= \{r(t)\}$.  Choose $U\in \mathscr G^a$ so that $ut\inv \in U$.  Replacing $U$ by $UU_0$ if necessary, we may assume that $r(s_i)\notin U\inv U$ for $i=1,\ldots,n$.  Then $Ut=ut\inv t=u$ and so $\chi_Ut=u$ by Proposition~\ref{bimodule}. On the other hand,  Proposition~\ref{bimodule}  provides $\chi_Us_i =0$, for $i=1,\ldots,n$. This completes the proof.
\end{proof}

An immediate corollary of the proposition is the following.

\begin{Cor}\label{cyclicmodule}
The module $KL_x$ is cyclic, namely $KL_x=K\mathscr G\cdot x$.  Consequently, if $V$ is a $KG_x$-module, then  $KL_x\otimes_{KG_x} V = K\mathscr G\cdot (x\otimes V)$.
\end{Cor}

It is easy to see that $\Ind_x$ is a functor from the category of $KG_x$-modules to the category of $K\mathscr G$-modules.   We now consider the restriction functor from $K\mathscr G$-modules to $KG_x$-modules, which is right adjoint to the induction functor.

\begin{Def}[Restriction]
For $x\in \mathscr G^0$, let $\mathscr N_x=\{U\in B(\mathscr G^0)\mid x\in U\}$.  If $V$ is a $K\mathscr G$-module, then define $\Res_x(V) = \bigcap_{U\in \mathscr N_x} UV$ where we view $V$ as a $K\mathscr G^a$-module via $Uv = \chi_Uv$ for $U\in \mathscr G^a$ and $v\in V$.
\end{Def}

In order to endow $\Res_x(V)$ with the structure of a $KG_x$-module, we need the following lemma.

\begin{Lemma}\label{welldefineactionofL_x}
Let $V$ be a $K\mathscr G$-module and put $W=\Res_x(V)$.  Then:
\begin{enumerate}
\item $K\mathscr G^a_x\cdot W=W$;
\item If $U\notin \til L_x$, then $UW=\{0\}$;
\item Let $U,U'\in \til L_x$ be such that $\nu(U)=\nu(U')$.  Then $Uw=U'w$ for all $w\in W$.
\end{enumerate}
\end{Lemma}
\begin{proof}
To prove (1), first observe that $\mathscr N_x\subseteq \mathscr G^a_x$ so $W\subseteq K\mathscr G^a_x\cdot W$.  For the converse, suppose that $U\in \mathscr G^a_x$ and $w\in W$.  Let $U_0\in \mathscr N_x$.  Then $U_0Uw = U(U\inv U_0U)w=Uw$ since $x\in U\inv U_0U$ and $w\in W$.  Since $U_0$ was arbitrary, we conclude that $Uw\in W$.

Turning to (2), suppose that $w\in W$ and $Uw\neq 0$.  Then $U\inv Uw\neq 0$.  Suppose $U_0\in \mathscr N_x$.  Then $U_0U\inv Uw=U\inv UU_0w = U\inv Uw$.  Hence the stabilizer in $B(\mathscr G^0)$ of $U\inv Uw$ is a proper filter containing the ultrafilter $\mathscr N_x$ and the element $U\inv U$.  We conclude that $U\inv U\in \mathscr N_x$ and so $U\in \til L_x$.

Next, we establish (3).  If $\nu(U)=\nu(U')$, then $U$ and $U'$ have a common lower bound $U_0\in \til L_x$ by Lemma~\ref{Lxasgerms}.  Hence, for any $w\in W$, we have $Uw=UU_0\inv U_0w= U_0w = U'U_0\inv U_0w = U'w$ as $U_0\inv U_0\in \mathscr N_x$.  This completes the proof.
\end{proof}

As a consequence of Lemmas~\ref{Lxasgerms} and~\ref{welldefineactionofL_x}, for $t\in L_x$ and $w\in \Res_x(V)$, there is a well-defined element $tw$ obtained by putting $tw=Uw$ where $U\in \mathscr G^a$ contains $t$. Trivially, the map $w\mapsto tw$ is linear.  Moreover, if $g\in G_x$ and $g\in U\in \mathscr G^a$, then $U\in \mathscr G^a_x$. Hence this definition gives $W$ the structure of a $KG_x$-module since the action of $\mathscr G^a_x$ on $W$ factors through its maximal group image $G_x$ by the aforementioned lemmas.  In particular, $xw=w$ for $w\in W$.  Let us now prove that if $V$ is a simple $K\mathscr G$-module and $\Res_x(V)\neq 0$, then it is simple.

\begin{Lemma}\label{issimple}
Let $V$ be a simple $K\mathscr G$-module.  Then the $KG_x$-module $\Res_x(V)$ is either zero or a simple $KG_x$-module.
\end{Lemma}
\begin{proof}
Set $W=\Res_x(V)$ and suppose that $0\neq w\in W$.  We need to show that $KG_x\cdot w=W$.  Let $w'\in W$.  Viewing $V$ as a $K\mathscr G^a$-module, we have $K\mathscr G^a\cdot w=V$ by simplicity of $V$. Hence $w'=(c_1U_1+\cdots +c_nU_n)w$ with $U_1,\ldots,U_n\in \mathscr G^a$.  Moreover, by Lemma~\ref{welldefineactionofL_x} we may assume $U_1,\ldots,U_n\in \til L_x$.  Let $t_i = \nu(U_i)$.  Choose $U\in \mathscr N_x$ so that $r(t_i)\in U$ implies $r(t_i)=x$.  Then $w' = Uw' = (c_1UU_1+\cdots +c_nUU_n)w$.  But $UU_i\in \til L_x$ if and only if $r(t_i) =x$, in which case $UU_i\in \mathscr G^a_x$.  Thus $w'\in K\mathscr G^a_x\cdot w = KG_x\cdot w$.  It follows that $W$ is simple.
\end{proof}

Next we establish the adjunction between $\Ind_x$ and $\Res_x$.  Since the functor $KL_x\otimes_{KG_x} (-)$ is left adjoint to $\mathrm{Hom}_{K\mathscr G}(KL_x,-)$, it suffices to show the latter is isomorphic to $\Res_x$.

\begin{Prop}\label{leftrightadjoint}
The functors $\Res_x$ and $\mathrm{Hom}_{K\mathscr G}(KL_x,-)$ are naturally isomorphic.  Thus $\Ind_x$ is the left adjoint of $\Res_x$.
\end{Prop}
\begin{proof}
Let $V$ be a $K\mathscr G$-module and put $W=\Res_x(V)$.  Define a homomorphism $\psi\colon \mathrm{Hom}_{K\mathscr G}(KL_x,V)\to W$ by $\psi(f) = f(x)$.    First note that if $U\in \mathscr N_x$, then $Uf(x)=f(Ux)=f(x)$ and so $f(x)\in W$.  Clearly $\psi$ is $K$-linear.  To see that it is $G_x$-equivariant, let $g\in G_x$ and choose $U\in \mathscr G^a_x$ with $g\in U$.  Then observe, using Proposition~\ref{bimodule}, that $\psi(gf)= (gf)(x) = f(xg) =f(g)= f(Ux) = Uf(x) = gf(x) = g\psi(f)$.  This shows that $\psi$ is a $KG_x$-morphism.  If $\psi(f)=0$, then $f(x)=0$ and so $f(KL_x)=0$ by Corollary~\ref{cyclicmodule}.  Thus $\psi$ is injective.  To see that $\psi$ is surjective, let $w\in W$ and define $f\colon KL_x\to V$ by $f(t) = tw$, for $t\in L_x$, where $tw$ is as defined after Lemma~\ref{welldefineactionofL_x}.  Then $f(x)=xw=w$.  It thus remains to show that $f$ is a $K\mathscr G$-morphism. To achieve this, it suffices to show that $f(Ut)=Utw$ for $U\in \mathscr G^a$.  Choose $U'\in \til L_x$ with $t\in U'$; so $Utw = UU'w$ by definition.  If $r(t)\notin U\inv U$, then $Ut=0$ (by Proposition~\ref{bimodule}) and $UU'\notin \til L_x$.  Thus $f(Ut)=0$, whereas $Utw=UU'w =0$ by Lemma~\ref{welldefineactionofL_x}.  On the other hand, if $r(t)\in U\inv U$ and say $d(s)=r(t)$ with $s\in U$, then $Ut=st$ and $st\in UU'\in \til L_x$.  Thus $f(Ut) = f(st)=(st)w$, whereas $Utw=UU'w = (st)w$.  This completes the proof that $f$ is a $K\mathscr G$-morphism and hence $\psi$ is onto.  It is clear that $\psi$ is natural.
\end{proof}

It turns out that $\Res_x\Ind_x$ is naturally isomorphic to the identity functor.

\begin{Prop}\label{composetoidentity}
Let $V$ be a $KG_x$-module.  Then $\Res_x\Ind_x(V)=x\otimes V$ is naturally isomorphic to $V$ as a $KG_x$-module.
\end{Prop}
\begin{proof}
Let $T$ be a transversal for $L_x/G_x$ with $x\in T$.   Because $T$ is a $KG_x$-basis for $KL_x$, it follows that $KL_x\otimes_{KG_x}V = \bigoplus_{t\in T}(t\otimes V)$.    We claim that $\Res_x\Ind_x(V)=x\otimes V$.  Indeed, if $U\in \mathscr N_x$, then $U(x\otimes v) = Ux\otimes v =x\otimes v$.  Conversely, suppose $w=t_1\otimes v_1+\cdots+t_n\otimes v_n$ belongs to $\Res_x\Ind_x(V)$.  Choose $U\in B(\mathscr G^0)$ so that $x\in U$ and $U\cap \{r(t_1),\cdots,r(t_n)\} \subseteq \{x\}$.  Then, by Proposition~\ref{bimodule}, we have $w=Uw\in x\otimes V$, establishing the desired equality.

Now $x\otimes V$ is naturally isomorphic to $V$ as a $KG_x$-module via the map $x\otimes v\mapsto v$  since if $g\in G$ and $U\in \mathscr G^a_x$ with $g\in U$, then $g(x\otimes v)=U(x\otimes v) = Ux\otimes v = g\otimes v = x\otimes gv$.
\end{proof}

A useful fact is that the induction functor is exact.  In general, $\Res_x$ is left exact but it need not be right exact.

\begin{Prop}
The functor $\Ind_x$ is exact, whereas $\Res_x$ is left exact.
\end{Prop}
\begin{proof}
Since $KL_x$ is a free $KG_x$-module, it is flat and hence $\Ind_x$ is exact.  Clearly $\Res_x = \mathrm{Hom}_{K\mathscr G}(KL_x,-)$ is left exact.
\end{proof}

Our next goal is to show that if $V$ is a simple $KG_x$-module, then the $K\mathscr G$-module $\Ind_x(V)$ is simple with a certain ``finiteness'' property, namely it is not annihilated by $\Res_x$.  Afterwards, we shall prove that all simple $K\mathscr G$-modules with this ``finiteness'' property are induced modules; this class of simple $K\mathscr G$-modules contains all the finite dimensional ones when $K$ is a field.  This is exactly what is done for inverse semigroups with finitely many idempotents in~\cite{mobius1,mobius2}.  Here the proof becomes more technical because the algebra need not be unital.  Also topology is used instead of finiteness arguments in the proof. The main idea is in essence that of~\cite{myirreps}: to exploit the adjunct relationship between induction and restriction.

The following definition will play a key role in constructing the finite dimensional irreducible representations of an inverse semigroup.

\begin{Def}[Finite index]
Let us say that an object $x\in \mathscr G^0$ has \emph{finite index} if its orbit is finite.
\end{Def}

\begin{Prop}\label{constructirred}
Let $x\in \mathscr G^0$ and suppose that $V$ is a simple $KG_x$-module.  Then $\Ind_x(V)$ is a simple $K\mathscr G$-module.
Moreover, if $K$ is a field, then $\Ind_x(V)$ is finite dimensional if and only if $x$ has finite index and $V$ is finite dimensional.  Finally, if $V$ and $W$ are non-isomorphic $KG_x$-modules, then $\Ind_x(V)\ncong \Ind_x(W)$.
\end{Prop}
\begin{proof}
We retain the notation above.   Let $T$ be a transversal for $L_x/G_x$ with $x\in T$.   Since $L_x/G_x$ is in bijection with the orbit $\mathscr O$ of $x$, the set $T$ is finite if and only if $x$ has finite index.  Because $T$ is a $KG_x$-basis for $KL_x$, it follows that $KL_x\otimes_{KG_x}V = \bigoplus_{t\in T}(t\otimes V)$.  In particular, $\Ind_x(V)$ is finite dimensional when $K$ is a field if and only if $T$ is finite and $V$ is finite dimensional.  This establishes the second statement.  We turn now to the proof of simplicity.

Suppose that $0\neq W$ is a $K\mathscr G$-submodule.  Then $\Res_x(W)$ is a $KG_x$-submodule of $\Res_x\Ind_x(V)\cong V$.  We claim that it is non-zero. Let $0\neq w\in W$.  Then $w=t_1\otimes v_1+\cdots +t_n\otimes v_n$ for some $v_1,\ldots, v_n\in V$ and $t_1,\ldots, t_n\in T$.  Moreover, $v_j\neq 0$ for some $j$.  By Proposition~\ref{createprojections}, we can find $U\in \mathscr G^a$ so that $\chi_Ut_j=x$ and $\chi_Ut_i=0$ for $i\neq j$.  Then $\chi_Uw=x\otimes v_j\neq 0$  belongs to $\Res_x(W)$.  Simplicity of $V$ now yields $\Res_x\Ind_x(V)=\Res_x(W)\subseteq W$.  Corollary~\ref{cyclicmodule} then yields \[\Ind_x(V)=K\mathscr G\cdot (x\otimes V) = K\mathscr G\cdot \Res_x\Ind_x(V)\subseteq K\mathscr G\cdot W\subseteq W,\] establishing the simplicity of $\Ind_x(V)$.

The final statement follows because $\Res_x\Ind_x$ is naturally equivalent to the identity functor.
\end{proof}

Next we wish to show that modules of the above sort obtained from distinct orbits are non-isomorphic.

\begin{Prop}\label{orbitunique}
Suppose that $x,y$ are elements in distinct orbits.  Then induced modules of the form $\Ind_x(V)$ and $\Ind_y(W)$ are not isomorphic.
\end{Prop}
\begin{proof}
Put $M=\Ind_x(V)$ and $N=\Ind_y(W)$. Proposition~\ref{composetoidentity} yields $\Res_x(M)\cong V\neq 0$. On the other hand, if $w=t_1\otimes v_1+\cdots+t_n\otimes v_n\in M$ is non-zero, then since $y\notin \{r(t_1),\ldots, r(t_n)\}$, we can find $U\in B(\mathscr G^0)$ so that $y\in U$ and $r(t_i)\notin U$, for $i=1,\ldots, n$.  Then $Uw=0$ by Proposition~\ref{bimodule}.  Thus we have $\Res_y(M)= 0$. Applying a symmetric argument to $N$ shows that $M\ncong N$.
\end{proof}

To obtain the converse of Proposition~\ref{constructirred}, we shall use Stone duality.   Also the inverse semigroup $\mathscr G^a$ will play a starring role since each $K\mathscr G$-module gives a representation of $\mathscr G^a$.  What we are essentially doing is imitating the theory for finite inverse semigroups~\cite{CP}, as interpreted through~\cite{mobius1,mobius2}, for ample groupoids; see also~\cite{myirreps}.   The type of simple modules which can be described as induced modules are what we shall term spectral modules.

%The following lemma concerns stabilizers of subspaces.
%
%\begin{Lemma}\label{stabilizersmove}
%Let $V$ be a $K\mathscr G$-module, which we view as a $K\mathscr G^a$-module via the projection $K\mathscr G^a\to K\mathscr G$.  Suppose $W\leq V$ is a non-zero subspace.  Then:
%\begin{enumerate}
%\item $\mathscr F_W = \{U\in B(\mathscr G^0)\mid UW=W\}$ is a proper filter;
%\item If $U\in \mathscr G^a$ with $U\inv U\in \mathscr F_W$, then $U\mathscr F_W = \mathscr F_{UW}$.
%\end{enumerate}
%\end{Lemma}
%\begin{proof}
%The first statement is trivial.  For the second suppose $U'UW = UW$. Then $U\inv U'UW = U\inv UW = W$ and so $U\inv U'U\in \mathscr F_W$ and hence $U'\in U\mathscr F_W$.  Conversely, if $U'\in U\mathscr F_W$, then $U\inv U'UW = W$ and so $U'UW=UW$.  This establishes (2).
%\end{proof}

\begin{Def}[Spectral module]
Let $V$ be a non-zero $K\mathscr G$-module.  We say that $V$ is a \emph{spectral module} if there is a point $x\in \mathscr G^0$ so that $\Res_x(V)\neq 0$.
\end{Def}

\begin{Rmk}
It is easy to verify that $\Res_x(K\mathscr G)\neq 0$ if and only if $x$ is an isolated point.  On the other hand, the cyclic module $KL_x$ satisfies $\Res_x(KL_x)=Kx$.  This shows that in general $\Res_x$ is not exact.  However, if $x$ is an isolated point of $\mathscr G^0$, then $\Res_x(V) = \delta_xV$ and so the restriction functor is exact in this case.
\end{Rmk}

Every induced module from a non-zero module is spectral due to the isomorphism $\Res_x\Ind_x(V)\cong V$.
Let us show that the spectral assumption is not too strong a condition.  In particular, we will establish that all finite dimensional modules over a field are spectral.    Recall that if $A$ is a commutative $K$-algebra, then the idempotent set $E(A)$ of $A$ is a generalized boolean algebra with respect to the natural partial order.  The join of $e,f$ is given by $e\vee f=e+f-ef$ and the relative complement by $e\setminus f=e-ef$.

\begin{Prop}
Let $V$ be a $K\mathscr G$-module with associated representation $\p\colon K\mathscr G\to \mathrm{End}_K(V)$.  Let $\alpha\colon \mathscr G^a\to \mathrm{End}_K(V)$ be the representation given by $U\mapsto \p(\chi_U)$.  Assume that $B=\alpha(B(\mathscr G^0))$ contains a primitive idempotent.  Then $V$ is spectral.  This occurs in particular if $B$ is finite or more generally satisfies the descending chain condition.
\end{Prop}
\begin{proof}
Let $A$ be the subalgebra spanned by $\alpha(B(\mathscr G^0))$; so $A=\p(K\mathscr G^0)$.  Then $\alpha\colon B(\mathscr G^0)\to E(A)$ is a morphism of generalized boolean algebras.  Indeed, we compute $\alpha(U\cup V) = \p(\chi_{U\cup V}) = \p(\chi_U)+\p(\chi_V)-\p(\chi_{U\cap V}) = \alpha(U)+\alpha(V)-\alpha(U)\alpha(V)=\alpha(U)\vee \alpha(V)$.  Thus $B$ is a generalized boolean algebra.  Stone duality provides a proper continuous map $\wh\alpha\colon \mathrm{Spec}(B)\to \mathscr G^0$.  So now let $e$ be a primitive idempotent of $B$.  Then $e^{\uparrow}$ is an ultrafilter on $B$ and $\chi_{e^{\uparrow}}\in \mathrm{Spec}(B)$ by Proposition~\ref{ultrafilterchar}.  Let $x=\wh{\alpha}(\chi_{e^{\uparrow}})$.  Then \[\Res_x(V) = \bigcap_{f\in e^{\uparrow}} fV= eV\neq 0\] completing the proof.
\end{proof}

The above proposition will be used to show that every finite dimensional representation over a field is spectral.
Denote by $M_n(K)$ the algebra of $n\times n$-matrices over $K$.  The following lemma is classical linear algebra.
\begin{Lemma}
Let $K$ be a field and $F\leq M_n(K)$ a semilattice.  Then we have $|F|\leq 2^n$.
\end{Lemma}
\begin{proof}
We just sketch the argument.  If $e\in F$, then $e^2=e$ and so the minimal polynomial of $e$ divides $x(x-1)$.  Thus $e$ is diagonalizable.  But commutative semigroups of diagonalizable matrices are easily seen to be simultaneously diagonalizable, so $F\leq K^n$.  But the idempotent set of $K$ is $\{0,1\}$, so $F\leq \{0,1\}^n$ and hence $|F|\leq 2^n$.
\end{proof}

\begin{Cor}\label{finiteidempotents}
Let $\p\colon K\mathscr G\to M_n(K)$ be a finite dimensional representation over a field $K$.  Then $\alpha(B(\mathscr G^0))$ is finite where $\alpha\colon B(\mathscr G^0)\to M_n(K)$ is given by $\alpha(U) = \p(\chi_U)$.  Consequently, every finite  dimensional (non-zero) $K\mathscr G$-module is spectral.
\end{Cor}

%The next result shows that modules for $\mathbb CG$ coming from actions on Hilbert space are spectral.
%
%\begin{Prop}
%Let $\p\colon \mathbb C\mathscr G\to \mathscr B(\mathfrak H)$ be a non-degenerate $\ast$-representation on a Hilbert space $\mathfrak H$.  Then $\mathfrak H$ is a spectral module.
%\end{Prop}
%\begin{proof}
%Since $C_c(\mathscr G^0)$ is the universal $C^*$-algebra of $C\mathscr G^0$, we have a $\ast$-representation $\p\colon C_c(\mathscr G^0)\to  \mathscr B(\mathfrak H)$.  Then $A=\p(C_c(\mathscr G^0))$ is a non-zero commutative $C^*$-algebra and by Gelfand duality we obtain a proper continuous map $\wh \p\colon \wh A\to \mathscr G^0$ where $\wh A\neq \emptyset$ is the Gelfand space of $A$.  Let $x$ be a point of the image of $\wh \p$ and put $\mathscr N_x=\{U\in B(\mathscr G^0\mid x\in U\}$.  Since $\wh \p$ is proper, the collection $\{\wh\p\inv(U)\mid U\in \mathscr N_x\}$ is a collection of compact sets satisfying the finite intersection condition.  It follows that there is a point $y\in \wh A$ so that $y\in \bigcap \wh\p\inv (U)$ for all $U\in \mathscr N_x$.
%\end{proof}

Now we establish the main theorem of this section.

\begin{Thm}\label{describeirreps}
Let $\mathscr G$ be an ample groupoid and fix $D\subseteq \mathscr G^0$ containing exactly one element from each orbit.  Then there is a bijection between spectral simple $K\mathscr G$-modules and pairs $(x,V)$ where $x\in D$ and $V$ is a simple $KG_x$-module (taken up to isomorphism).  The corresponding simple $K\mathscr G$-module is $\Ind_x(V)$.  When $K$ is a field, the finite dimensional simple $K\mathscr G$-modules correspond to those pairs $(x,V)$ where $x$ is of finite index and $V$ is a finite dimensional simple $KG_x$-module.
\end{Thm}
\begin{proof}
Proposition~\ref{constructirred} and Proposition~\ref{orbitunique} yield that the modules described in the theorem statement form a set of pairwise non-isomorphic spectral simple $K\mathscr G$-modules.  It remains to show that all spectral simple $K\mathscr G$-modules are of this form.  So let $V$ be a spectral simple $K\mathscr G$-module and suppose $\Res_x(V)\neq 0$.  Then $\Res_x(V)$ is a simple $KG_x$-module by Lemma~\ref{issimple}.  By the adjunction between induction and restriction, the identity map on $\Res_x(V)$ gives rise to a non-zero $K\mathscr G$-morphism $\psi\colon \Ind_x\Res_x(V)\to V$.  Since $\Ind_x\Res_x(V)$ is simple by Proposition~\ref{constructirred} and $V$ is simple by hypothesis, it follows that $\psi$ is an isomorphism by Schur's Lemma.  This completes the proof of the first statement since the induced modules depend only on the orbit up to isomorphism.  The statement about finite dimensional simple modules is a consequence of Proposition~\ref{constructirred} and Corollary~\ref{finiteidempotents}.
\end{proof}

\subsection{Irreducible representations of inverse semigroups}
Fix now an inverse semigroup $S$ and let $\mathscr G(S)$ be the universal groupoid of $S$.   If $\p\in \wh {E(S)}$ has finite index in $\mathscr G(S)$, then we shall call $\p$ a \emph{finite index character} of $E(S)$ in $S$. This notion of index really depends on $S$.  Notice that the orbit of $\p$ in $\mathscr G(S)$ is precisely the orbit of $\p$ under the spectral action of $S$ on $\wh {E(S)}$.  If $E(S)$ is finite, then of course $\wh {E(S)}=E(S)$ and all characters have finite index.  If $\p\in \wh{E(S)}$, then $S_{\p} = \{s\in S\mid s\p =\p\}$ is an inverse subsemigroup of $S$ and one easily checks that the isotropy group $G_{\p}$ of $\p$ in $\mathscr G(S)$ is precisely the maximal group image of $S_{\p}$ since if $s,s'\in S_{\p}$ and $t\leq s,s'$ with $\p\in D(t^*t)$ and $t\in S$, then $t\in S_{\p}$.  This allows us to describe the finite dimensional irreducible representations of an inverse semigroup without any explicit reference to $\mathscr G(S)$.  So without further ado, we state the classification theorem for finite dimensional irreducible representations of inverse semigroups, thereby generalizing the classical results for inverse semigroups with finitely many idempotents~\cite{CP,mobius1,mobius2,oknisemigroupalgebra}.

\begin{Thm}\label{inversedescribereps}
Let $S$ be an inverse semigroup and $K$ a field.  Fix a set $D\subseteq \wh{E(S)}$ containing exactly one finite index character from each orbit of finite index characters under the spectral action of $S$ on $\wh{E(S)}$.  Let $S_{\p}$ be the stabilizer of $\p$ and set $G_{\p}$ equal to the maximal group image of $S_{\p}$.    Then there is a bijection between finite dimensional simple $KS$-modules and pairs $(\p,V)$ where $\p\in D$ and $V$ is a finite dimensional simple $KG_{\p}$-module (considered up to isomorphism).
\end{Thm}
\begin{proof}
This is immediate from Theorem~\ref{describeirreps} and the above discussion.
\end{proof}

\begin{Rmk}
That there should be a theorem of this flavor was first suggested in unpublished joint work of S.~Haatja, S.~W.~Margolis and the author from 2002.
\end{Rmk}

Let us draw some consequences. First we give necessary and sufficient conditions for an inverse semigroup to have enough finite dimensional irreducible representations to separate points.  Then we provide examples showing that the statement cannot really be simplified.

\begin{Cor}\label{seppoints}
An inverse semigroup $S$ has enough finite dimensional irreducible representations over $K$ to separate points if and only if:
\begin{enumerate}
\item The characters of $E(S)$ of finite index in $S$ separate points of $E(S)$;
\item For each $e\in E(S)$ and each $e\neq s\in S$ so that $s^*s=e=ss^*$, there is a character $\p$ of finite index in $S$ so that $\p(e)=1$ and either:
\begin{enumerate}
\item $s\p\neq \p$; or
\item $s\p =\p$ and there is a finite dimensional irreducible representation $\psi$ of $G_{\p}$ so that $\psi([s,\p])\neq 1$.
\end{enumerate}
\end{enumerate}
\end{Cor}
\begin{proof}
Suppose first that $S$ has enough finite dimensional irreducible representations to separate points and that $e\neq f$ are idempotents of $S$. Choose a finite dimensional simple $KS$-module $W=KL_{\p}\otimes_{KG_{\p}} V$ with $\p$ a finite index character and such that $e$ and $f$ act differently on $W$.  Recalling that $x\mapsto \chi_{D(x)}$ for $x\in E(S)$ under the isomorphism $KS\to K\mathscr G$, it follows from Proposition~\ref{bimodule} that, for $t\in L_{\p}$,
\begin{equation}\label{idempotentaction}
xt = \begin{cases} t & r(t)(x)=1\\
0 & r(t)(x)=0.\end{cases}
\end{equation}
Therefore, in order for $e$ and $f$ to act differently on $W$, there must exist $t\in L_{\p}$ with $r(t)=\rho$ a finite index character such that $\rho(e)\neq \rho(f)$.

Next suppose that $e\neq s$ and $s^*s=e=ss^*$.  By assumption there is a finite dimensional simple $KS$-module $W=KL_{\p}\otimes_{KG_{\p}} V$, with $\p$ a finite index character, where $s$ and $e$ act differently. By \eqref{idempotentaction} there must exist $t\in L_{\p}$ and $v\in V$ so that $\rho=r(t)$ satisfies $\rho\in D(e)$ and $s(t\otimes v)\neq e(t\otimes v) = t\otimes v$.  Since $\rho$ has finite index, if $s\rho\neq \rho$ then we are done.  So assume $s\rho=\rho$.

Recall that under the isomorphism of algebras $KS\to K\mathscr G(S)$, we have that $s\mapsto \chi_{(s,D(e))}$. Since $et\neq 0$ implies $st\neq 0$ (as $s^*s=e$), there must exist $y\in L_{\p}$ so that $yt\inv\in (s,D(e))$ and moreover $st = y$ in $KL_{\p}$ by Proposition~\ref{bimodule}.  We must then have $yt\inv = [s,\rho]\in G_{\rho}$ as $s\rho=\rho$.  Now $st = y=t(t\inv y) =t(t\inv [s,\rho]t)$ and so $t\otimes v\neq s(t\otimes v) = t\otimes t\inv [s,\rho]tv$.  Thus if we make $V$ a (simple) $KG_{\rho}$-module via $gv = (t\inv gt)v$, then $[s,\rho]$ does not act as the identity on this module.  This completes the proof of necessity.

Let us now proceed with sufficiency.  First we make an observation.  Let $\p$ be a character of finite index with associated finite orbit $\mathscr O$.  Let $V$ be the trivial $KG_{\p}$-module.  It is routine to verify using Proposition~\ref{bimodule} that $KL_{\p}\otimes _{KG_{\p}} V$ has a basis in bijection with $\mathscr O$ and $S$ acts on the basis by restricting the action of $S$ on $\wh{E(S)}$ to $\mathscr O$.  We call this the \emph{trivial representation associated to $\mathscr O$}.

Suppose $s,t\in S$ with $s\neq t$.  Assume first that $s^*s\neq t^*t$ and let $\p$ be a finite index character with $\p(s^*s)\neq \p(t^*t)$.   Then in the trivial representation associated to the orbit of $\p$, exactly one of $s$ and $t$ is defined on $\p$ and so this finite dimensional irreducible representation separates $s$ and $t$.  A dual argument works if $ss^*\neq tt^*$.

So let us now assume that $s^*s=t^*t$ and $ss^*=tt^*$.  Then it suffices to separate $s^*s$ from $t^*s$ in order to separate $s$ and $t$.  So we are left with the case that $s^*s=e=ss^*$ and $s\neq e$.  We have two cases.  Suppose first we can find a finite index character $\p$ with $s\p\neq \p$.  Again, the trivial representation associated to the orbit of $\p$ separates $s$ and $e$.

Suppose now that there is a finite index character $\p$ with $\p(e)=1$ and $s\p=\p$ and a finite dimensional simple $KG_{\p}$-module $V$ so that $[s,\p]$ acts non-trivially on $V$.   It is then easy to see  using Proposition~\ref{bimodule} that $s(\p\otimes v) = [s,\p]\otimes v=\p\otimes [s,\p]v$ since $[s,\p]\in (s,D(e))$.  Thus $s$ acts non-trivially on $KL_{\p}\otimes _{KG_{\p}}V$, completing the proof.
\end{proof}

An immediate consequence of this corollary is the following folklore result.
\begin{Cor}
Let $S$ be an inverse semigroup with finitely many idempotents and $K$ a field.  Then there are enough finite dimensional irreducible representations of $S$ over $K$ to separate points if and only if each maximal subgroup of $S$ has enough finite dimensional irreducible representations to separate points.
\end{Cor}

As a first example, consider the bicyclic inverse monoid, presented by $B=\langle x\mid x^*x=1\rangle$.  Any non-degenerate finite dimensional representation of $B$ must be by invertible matrices since left invertibility implies right invertibility for matrices.  Hence one cannot separate the idempotents of $B$ by finite dimensional irreducible representations of $B$ over any field.  To see this from the point of view of Corollary~\ref{seppoints}, we observe that $\wh{E(B)}$ is the one-point compactification of the natural numbers.  Namely, if $F$ is a filter on $E(B)$, then either it has a minimum element $x^n(x^*)^n$, and hence is a principal filter, or it contains all the idempotents (which is the one-point compactification).  All the principal filters are in a single (infinite) orbit.  The remaining filter is in a singleton orbit with isotropy group $\mathbb Z$.  It obviously separates no idempotents.

Let us next give an example to show that there can be enough finite dimensional irreducible representations of an inverse semigroup to separate points, and yet there can be a finite index character $\p$ so that the isotropy group $G_{\p}$ does not have enough irreducible representations to separate points.  Let $K=\mathbb C$.  Then any finite inverse semigroup has enough finite dimensional irreducible representations to separate points, say by the above corollary.  Hence any residually finite inverse semigroup has enough finite dimensional irreducible representations over $\mathbb C$ to separate points.  On the other hand, the maximal group image $G$ of an inverse semigroup $S$ is the isotropy group of the trivial character that sends all idempotents to $1$, which is a singleton orbit of $\wh{E(S)}$.  Let us construct a residually finite inverse semigroup whose maximal group image does not have any non-trivial finite dimensional representations.

A well-known result of Mal'cev says that a finitely generated group $G$ with a faithful finite dimensional representation over $\mathbb C$ is residually finite.  Since any representation of a simple group is faithful and an infinite simple group is trivially not residually finite, it follows that finitely generated infinite simple groups have no non-trivial finite dimensional representations over $\mathbb C$.  An example of such a group is the famous Thompson's group $V$, which is a finitely presented infinite simple group~\cite{Thompsongroup}.

In summary, if we can find a residually finite inverse semigroup whose maximal group image is a finitely generated infinite simple group, then we will have found the example we are seeking. To construct our example, we make use of the Birget-Rhodes expansion~\cite{BR--exp}.
Let $G$ be any group and let $E$ be the semilattice of finite subsets of $G$ ordered by reverse inclusion (so the meet is union).  Let $G$ act on $E$ by left translation, so  $gX=\{gx\mid x\in X\}$, and form the semidirect product $E\rtimes G$.  Let $S$ be the inverse submonoid of $E\rtimes G$ consisting of all pairs $(X,g)$ so that $1,g\in X$.  This is an $E$-unitary (in fact $F$-inverse) monoid with maximal group image $G$ and identity $(\{1\},1)$.  It is also residually finite.  To see this, we use the well-known fact that an inverse semigroup all of whose $\R$-classes are finite is residually finite (the right Sch\"utzenberger representations on the $\R$-classes separate points).  Hence it suffices to observe that $(X,g)(X,g)^* = (X,1)$ and so the $\R$-class of $(X,g)$ consists of all elements of the form $(X,h)$ with $h\in X$, which is a finite set.

Let us observe that Mal'cev's result immediately implies that a finitely generated group has enough finite dimensional irreducible representations over $\mathbb C$ to separate points if and only if it is residually finite.  One direction here is trivial. For the non-trivial direction, suppose $G$ has enough finite dimensional irreducible representations over $\mathbb C$ to separate points and suppose $g\neq 1$.  Then $G$ has a finite dimensional irreducible representation $\p\colon G\to GL_n(\mathbb C)$ so that $\p(g)\neq 1$.  But $\p(G)$ is a finitely generated linear group and so residually finite by Mal'cev's theorem.  Thus we can find a homomorphism $\psi\colon \p(G)\to H$ with $H$ a finite group and $\psi(\p(g))\neq 1$.  Mal'cev's theorem immediately extends to inverse semigroups.

\begin{Prop}
Let $S$ be a finitely generated inverse subsemigroup of $M_n(\CC)$.  Then $S$ is residually finite.
\end{Prop}
\begin{proof}
Set $V=\mathbb C^n$.  We know that $E(S)$ is finite by Lemma~\ref{finiteidempotents} and hence each maximal subgroup is finitely generated by Corollary~\ref{fgmaxsubgroup}.  It follows that each maximal subgroup is residually finite by Mal'cev's theorem since the maximal subgroup $G_e$ is a faithful group of linear automorphisms of $eV$ for $e\in E(S)$.  But it is well known and easy to prove that if $S$ is an inverse semigroup with finitely many idempotents, then $S$ is residually finite if and only if all its maximal subgroups are residually finite. Indeed, for the non-trivial direction observe that the right Sch\"utzenberger representations of $S$ separate points into partial transformation wreath products of the form $G\wr T$ with $T$ a transitive faithful inverse semigroup of partial permutations of a finite set and $G$ a maximal subgroup of $S$.  But such a wreath product is trivially residually finite when $G$ is residually finite.
\end{proof}

Now the exact same proof as the group case establishes the following result.

\begin{Prop}
Let $S$ be a finitely generated inverse semigroup.  Then $S$ has enough finite dimensional irreducible representations over $\mathbb C$ to separate points if and only if $S$ is residually finite.
\end{Prop}

For the remainder of the section we take $K$ to be a commutative ring with unit.
We now characterize the spectral $KS$-modules in terms of $S$.  In particular, we shall see that if $E(S)$ satisfies the descending chain condition, then all non-zero $KS$-modules are spectral and so we have a complete parameterization of all simple $KS$-modules.

\begin{Prop}\label{whenisspectral}
Let $S$ be an inverse semigroup and let $V$ be a non-zero $KS$-module.  Then $V$ is a spectral $K\mathscr G(S)$-module if and only if there exists $v\in V$ so that $fv= v$ for some idempotent $f\in E(S)$ and, for all $e\in E(S)$, one has $ev\neq 0$ if and only if $ev=v$.  In particular, if $\p\colon S\to \mathrm{End}_K(V)$ is the corresponding representation and $\p(E(S))$ contains a primitive idempotent (for instance, if it satisfies the descending chain condition), then $V$ is spectral.
\end{Prop}
\begin{proof}
Recall that $e\mapsto \chi_{D(e)}$ under the isomorphism of $KS$ with $K\mathscr G(S)$.
Suppose first $V$ is spectral and let $\theta\in \wh{E(S)}$ so that $\Res_{\theta}(V)\neq 0$.  Fix $0\neq v\in \Res_{\theta}(V)$.  If $\theta(f)\neq 0$, then $D(f)\in \mathscr N_{\theta}$ and so $fv=v$.  Suppose that $e\in E(S)$ with $ev\neq 0$.    Then $\{U\in B(\wh {E(S)})\mid Uev=ev\}$ is a proper filter containing $\mathscr N_{\theta}$ and $D(e)$.  Since $\mathscr N_{\theta}$ is an ultrafilter, we conclude $D(e)\in \mathscr N_{\theta}$ and so $ev=D(e)v=v$.

Conversely, suppose there is an element $v\in V$ so that $fv=v$ some $f\in E(S)$ and $ev\neq 0$ if and only if $ev=v$ for all $e\in E(S)$; in particular $v\neq 0$.  Let $A=\p(KE(S))$ where $\p\colon KS\to \mathrm{End}_k(V)$ is the associated representation.  We claim that the set $B$ of elements $e\in E(A)$ so that $ev\neq 0$ implies $ev=v$ is a generalized boolean algebra containing $\p(E(S))$.  It clearly contains $0$.  Suppose $e,f\in B$ and $efv\neq 0$.  Then $ev\neq 0\neq fv$ so $efv=v$.   On the other hand, assume $(e+f-ef)v = (e\vee f)v\neq 0$.  Then at least one of $ev$ or $fv$ is non-zero.  If $ev\neq 0$ and $fv=0$, then we obtain $(e\vee f)v = ev+fv-efv=v$.  A symmetric argument applies if $ev=0$ and $fv\neq 0$.  Finally, if $ev\neq 0\neq fv$, then $(e\vee f)v = ev+fv-efv = v$.  To deal with relative complements, suppose $e,f\in B$ and $(e-ef)v=(e\setminus f)v \neq 0$.  Then $ev\neq 0$ and so $ev=v$.  Therefore, $(e-ef)v=v-fv$.  If $fv\neq 0$, then $fv=v$ and so $(e-ef)v=0$, a contradiction.  Thus $fv=0$ and $(e\setminus f)v=v$.  Since $E(S)$ generates $B(\wh {E(S)})$ as a generalized boolean algebra via the map $e\mapsto D(e)$, it follows that if $U\in B(\wh {E(S)})$ and $Uv\neq 0$, then $Uv=v$.  Let $\mathscr F=\{U\in B(\wh {E(S)})\mid Uv=v\}$.   Clearly,  $\mathscr F$ is a proper filter. We claim that it is an ultrafilter.  Indeed, suppose that $U'\notin \mathscr F$.  Then $U'v=0$ and so $(U\setminus U')v = Uv-UU'v=v$.  Thus $U\setminus U'\in \mathscr F$ and so $\emptyset = U'\cap (U\setminus U')$ shows that the filter generated by $U'$ and $\mathscr F$ is not proper.  Thus $\mathscr F$ is an ultrafilter on $B(\wh {E(S)})$ and hence is of the form $\mathscr N_{\theta}$ for a unique element $\theta\in \wh{E(S)}$.  It follows $v\in \Res_{\theta}(V)$.

For the final statement,  suppose that $\p(f)\in \p(E(S))$ is primitive and $0\neq v\in fV$. Then, for all $e\in E(S)$, $efv=ev\neq 0$ implies $\p(ef)\neq 0$ and so $\p(ef)=\p(f)$ by primitivity.  Thus $ev = efv=fv=v$.
\end{proof}

It turns out that if every idempotent of an inverse semigroup is central, then every simple $KS$-module is spectral.

\begin{Prop}
Let $S$ be an inverse semigroup with central idempotents.  Then every simple $KS$-module is spectral.
\end{Prop}
\begin{proof}
Let $V$ be a simple $KS$-module and suppose $e\in E(S)$.  Since $e$ is central, it follows that $eV$ is $KS$-invariant and hence $eV=V$, and so $e$ acts as the identity, or $eV=0$, whence $e$ acts as $0$.  Thus $V$ is spectral by Proposition~\ref{whenisspectral}
\end{proof}

There are other classes of inverse semigroups all of whose modules are spectral (and hence for which we have a complete list of all simple modules).

\begin{Prop}
Let $S$ be an inverse semigroup such that $E(S)$ is isomorphic to $(\mathbb N,\geq)$.   Then every non-zero $KS$-module is spectral.
\end{Prop}
\begin{proof}
Suppose $E(S)=\{e_i\mid i\in \mathbb N\}$ with $e_ie_j = e_{\max\{i,j\}}$.
Let $V$ be a $KS$-module.  If $eV=V$ for all $e\in E(S)$, then trivially $V$ is spectral.  Otherwise, we can find $n>0$ minimum so that $e_nV\neq V$.  Then $V=e_nV\oplus (1-e_n)V$ and $(1-e_n)V\neq 0$.  Choose a non-zero vector $v$ from $(1-e_n)V$.  We claim \[e_iv=\begin{cases} v & i<n\\ 0 & i\geq n.\end{cases}\]  It will then follow that $V$ is a spectral $KS$-module by Proposition~\ref{whenisspectral}.  Indeed, if $i<n$, then $e_i$ acts as the identity on $V$ by choice of $n$. On the other hand, if $i\geq n$, then $e_i(1-e_n) = e_i-e_ie_n = e_i-e_i=0$.  This completes the proof.
\end{proof}

Putting it all together we obtain the following theorem.

\begin{Thm}\label{inversedescribereps2}
Let $S$ be an inverse semigroup and $K$ a commutative ring with unit.  Fix a set $D\subseteq \wh{E(S)}$ containing exactly one character from each orbit of the spectral action of $S$ on $\wh{E(S)}$.  Let $S_{\p}$ be the stabilizer of $\p$ and set $G_{\p}$ equal to the maximal group image of $S_{\p}$.    Then there is a bijection between simple $KS$-modules $V$ so that there exists $v\in V$ with \[\emptyset\neq \{e\in E(S)\mid ev=v\} = \{e\in E(S)\mid ev\neq 0\}\] and pairs $(\p,W)$ where $\p\in D$ and $W$ is a simple $KG_{\p}$-module (considered up to isomorphism).   This in particular, describes all simple $KS$-modules if the idempotents of $S$ are central or from a descending chain isomorphic to $(\mathbb N,\geq)$.
\end{Thm}

For example, if $B$ is the bicyclic monoid, then the simple $KB$-modules are the simple $K\mathbb Z$-modules and the representation of $B$ on the polynomial ring $K[x]$ by the unilateral shift.
At the moment we do not have an example of an inverse semigroup $S$ and a simple $KS$-module that is not spectral.
By specializing to inverse semigroups with descending chain condition on idempotents, we obtain the following generalization of Munn's results~\cite{CP,oknisemigroupalgebra}.

\begin{Cor}\label{inversedescribereps3}
Let $S$ be an inverse semigroup satisfying descending chain condition on idempotents and let $K$ be a commutative ring with unit. Fix a set $D\subseteq E(S)$ containing exactly one idempotent from each $\D$-class.  Then there is a bijection between simple $KS$-modules and pairs $(e,V)$ where $e\in D$ and $V$ is a simple $KG_e$-module (considered up to isomorphism).  The corresponding $KS$-module is finite dimensional if and only if the $\D$-class of $e$ contains finitely many idempotents and $V$ is finite dimensional.
\end{Cor}

\bibliographystyle{abbrv}
\bibliography{standard2}
\end{document}